\documentclass[12pt,reqno]{amsart}

\usepackage{fullpage}
\usepackage{mathdots}
\usepackage{amssymb,amsfonts}
\usepackage[all,arc]{xy}
\usepackage{enumerate}
\usepackage{mathrsfs}
\usepackage{color}   
\usepackage{hyperref}
\hypersetup{
    colorlinks=true, 
    linktoc=all,     
    linkcolor=blue,  
}
\usepackage{amsthm}
\usepackage{amsmath}
\usepackage{graphicx}
\usepackage{pgf,tikz}
\usetikzlibrary{positioning}

\newcommand{\scm}{\mathcal{M}}

\newcommand{\bc}{\mathbb C}

\newcommand{\bz}{\mathbb Z}

\newcommand{\bq}{\mathbb Q}
\newcommand{\br}{\mathbb R}

\newcommand{\la}{\langle}
\newcommand{\ra}{\rangle}

\newcommand{\bs}{\backslash}

\newcommand{\al}{\alpha}
\newcommand{\be}{\beta}

\newcommand{\lam}{\lambda}

\DeclareMathOperator{\GL}{GL}

\DeclareMathOperator{\BZL}{BZL}

\newcommand{\fg}{\mathfrak g}

\newcommand{\fb}{\mathfrak b}

\newcommand{\fsl}{\mathfrak{sl}}
\newcommand{\fso}{\mathfrak{so}}

\newcommand{\sco}{\mathcal{O}}
\newcommand{\scb}{\mathcal B}

\newtheorem{Thm}{Theorem}[section]
\newtheorem{Prop}[Thm]{Proposition}
\newtheorem{Lem}[Thm]{Lemma}
\newtheorem{Cor}[Thm]{Corollary}

\theoremstyle{definition}
\newtheorem{Def}[Thm]{Definition}

\theoremstyle{remark}
\newtheorem{Rem}[Thm]{Remark}
\newtheorem{Ex}[Thm]{Example}

\theoremstyle{definition}

\theoremstyle{definition}
\newtheorem*{SA}{Stability Assumption}

\title{On the crystal graph description of the stable Weyl group multiple Dirichlet series}
\author{Yuanqing Cai}
\address{Department of Mathematics, Kyoto University, Kitashirakawa Oiwake-cho, Sakyo-ku, Kyoto 606-8502, Japan}
\email{cai@math.kyoto-u.ac.jp}
\date\today

\begin{document}

\begin{abstract}
For a semisimple Lie algebra admitting a good enumeration, we prove a parametrization for the elements in its Weyl group.
As an application, we give a coordinate-free comparison between the crystal graph description (when it is known) and the Lie-theoretic description of the  Weyl group multiple Dirichlet series in the stable range.
\end{abstract}

\maketitle


\section{Introduction}

The purpose of this paper is to give a comparison between two descriptions of the \textit{Weyl group multiple Dirichlet series}.  Weyl group multiple Dirichlet series (associated to a root system $\Phi$ and a positive integer $n$) are Dirichlet series in $r$ complex variables which initially converge on a cone in $\bc^r$, possess analytic continuation to a meromorphic function on the whole complex space, and satisfy functional equations whose action on $\bc^r$ is isomorphic to the Weyl group of $\Phi$. They arise as Whittaker coefficients of Eisenstein series on covers of reductive groups and have applications in analytic number theory.

We begin by describing the shape of the Weyl group multiple Dirichlet series. Given a number field $F$ containing the $2n$-th roots of unity and a finite set of places $S$ of $F$ (chosen with certain restrictions), let $\sco_S$ denote the ring of $S$-integers in $F$ and $\sco_S^\times$ the units in this ring. Then to any $r$-tuple of nonzero $\sco_S$ integers $\boldsymbol{m}=(m_1,\cdots,m_r)$, we associate a Weyl group multiple Dirichlet series in $r$ complex variables $\boldsymbol{s}=(s_1,\cdots,s_r)$ of the form
\begin{equation}\label{eq:def of wgmds}
Z_\Psi(s_1,\cdots,s_r;m_1,\cdots,m_r)
=Z_\Psi(\boldsymbol{s}; \boldsymbol{m})
=\sum_{{\bf c}=(c_1,\cdots,c_r)\in(\sco_S/\sco_S^\times)^r}\dfrac{H^{(n)}(\boldsymbol{c};\boldsymbol{m})\Psi(\boldsymbol{c})}{|c_1|^{2s_1}\cdots |c_r|^{2s_r}},
\end{equation}
where the coefficients $H^{(n)}(\boldsymbol{c};\boldsymbol{m})$ carry the main arithmetic content. The function $\Psi(\boldsymbol{c})$ guarantees the numerator of our series is well-defined up to $\sco_S^\times$ units. Finally $|c_i|=|c_i|_S$ denotes the norm of the integer $c_i$ as a product of local norms on $F_S=\prod_{\nu\in S}F_\nu$.

The coefficients $H^{(n)}(\boldsymbol{c};\boldsymbol{m})$ are not multiplicative, but nearly so and can  be reconstructed from coefficients of the form
\[
H^{(n)}(p^{\boldsymbol{k}};p^{\boldsymbol{l}}):=H^{(n)}(p^{k_1},\cdots, p^{k_r};p^{l_1},\cdots, p^{l_r}),
\]
where $p$ is a fixed prime in $\sco_S$, $\boldsymbol{k}=(k_1,\cdots,k_r)$, $\boldsymbol{l}=(l_1,\cdots,l_r)$, $k_i=\mathrm{ord}_p(c_i)$ and $l_i=\mathrm{ord}_p(m_i)$.

There are at least five approaches to defining these prime-power contributions so that $Z_\Psi(\boldsymbol{s}; \boldsymbol{m})$ admits analytic continuation to $\bc^r$ and satisfies the desired functional equations.

\begin{enumerate}
\item\label{case:stable} When $n$ is sufficiently large (see Section \ref{sec:stable case} below),  the $p$-parts admit a simple Lie-theoretic definition.    This is done in \cite{BBF-untwisted-stable, BBF-twisted}. For a fixed $\boldsymbol{l}$, the $\boldsymbol{k}$'s such that $H^{(n)}(p^{\boldsymbol{k}};p^{\boldsymbol{l}})\neq 0$ have a bijection with the Weyl group $W(\Phi)$.

\item\label{case:cg} Chinta-Gunnells \cite{CG-constructing} use a remarkable action of the Weyl group to define the coefficients $H^{(n)}(p^{\boldsymbol{k}};p^{\boldsymbol{l}})$ as an average over elements of the Weyl group for any root system $\Phi$ and any integer $n\geq 1$, from which functional equations and analytic continuation of the series $Z$ follows.

\item\label{case:crystal} For $\Phi$ of type $A$ and any $n\geq 1$,  the prime-power coefficients are defined as a sum over basis vectors in a highest weight representation for $\GL(r+1,\bc)$ associated to the fixed $r$-tuple $\boldsymbol{l}$ (\cite{BBF-combinatorial, BBF-crystal}). By choosing a   nice decomposition of the longest element, this can also be described using a combinatorial model for highest weight representations -- the Berenstein-Zelevinsky-Littelmann patterns. The same is carried out for type $C$ and odd $n$  in  \cite{FZ-orthogonal} (an inductive formula is also established for even degree covers).  A conjecture of this form for type $D$ is stated in  \cite{CG-littelmann}.
\item The prime-power coefficients can also be interpreted as values of Whittaker functions on metaplectic groups; see \cite{mcnamara,McNamara16}.
\item The prime-power coefficients can also be defined via ice models and have interesting connections to the Yang-Baxter equation (\cite{BBF-ice}).

\end{enumerate}

It is widely believed that these definitions agree. In this paper we focus on the first three.  The equivalence between (\ref{case:stable})  and (\ref{case:cg})  is shown in \cite{Fri}, generalizing the approach in \cite{CFG}. In type A, the equivalence between (\ref{case:cg}) and (\ref{case:crystal}) follows from \cite{mcnamara} and \cite{CO-metaplectic}, by interpreting both descriptions as values of unramified Whittaker functions. The equivalence between (\ref{case:stable}) and (\ref{case:crystal})  can be found in \cite{BBF-twisted} Section 8 for type A and \cite{BBF-typec} Section 4 for type C. In both cases, the proofs rely on explicit realization of the root systems in $\br^r$.

In \cite{BBF-typec}, the authors ask if a coordinate-free proof (i.e. without explicit realization of the root systems in $\br^r$) for the equivalence between (\ref{case:stable}) and (\ref{case:crystal}) exists. This is the question we would like to address.

Assume that $n$ is sufficiently large. Our comparison is based on the following observations.
\begin{itemize}
\item The $\BZL$-patterns having nontrivial contribution in the crystal graph description correspond to the orbit of the highest weight vector under the Weyl group.  This provides a natural candidate for such a bijection.
    \item In the crystal graph description, one first decorates each pattern, and the contribution is expressed in terms of $n$-th order Gauss sums according to the decoration. One easily sees that for a pattern to have nontrivial contribution, it must be a ``stable pattern'' (see Definition \ref{def:stable and unstable}). We show that decorations for stable patterns must have much nicer shapes (see Lemma \ref{lem:unstable decoration}).
\end{itemize}

 Our goal is to establish a bijection between the set of possible decorations and the Weyl group directly. By doing this, we can give an explicit bijection between the set of stable patterns and the Weyl group (see Proposition \ref{prop:bijection}), without passing to the orbit of the highest weight vector. The comparison between the crystal graph description and the Lie-theoretic description of the Weyl group multiple Dirichlet series is very natural in this setting.

In the process of confirming this bijection, we discover that the nice decomposition of the longest element plays an important role. A bijection between a set of decorations and the Weyl group is also established for root systems admitting \textit{good enumerations} (see Section \ref{sec:good enumeration} below), not just for root systems of type $A$ and $C$.

The rest of the paper is organized as follows. In Section \ref{sec:braidless and minimal}, we describe the parametrization of the Weyl group for certain root systems as a set of decorations. Our main result in this section is Proposition \ref{prop:main result in weyl}. This section is also of independent interest. The cases required for applications in Section \ref{sec:WGMDS} are highlighted in Section \ref{sec:examples} and the rest is given in Appendix \ref{app:examples}. The study of $\BZL$-patterns is carried out in Section \ref{sec:BZL patterns}.   We recall the definition and properties of $\BZL$-patterns. We define stable and unstable patterns, and prove various properties. In particular, we give a natural bijection between the Weyl group and the set of stable patterns.  Section \ref{sec:WGMDS} is devoted to Weyl group multiple Dirichlet series. We first review  the Lie-theoretic description in the stable range and the known cases of the crystal graph description.  Our main result, i.e. the comparison between these two descriptions in the stable range, is done in Section \ref{sec:comparison}.  The comparison can be done term by term canonically.

Some of the results in this paper are only stated for type $A$ and $C$. We try to minimize the use of case-by-case check in the proofs as some of them actually work in a more general setting. Our results also suggest what the crystal graph description should look like for general root systems, at least for stable patterns (see Remark \ref{rem:general description}).

\subsection*{Acknowledgement}
The author would like to thank the referee for very detailed and helpful comments for an earlier version of the paper, which correct several mistakes and improve the presentation of the paper.
Part of this work was carried out when the author was a postdoctoral fellow at the Weizmann Institute of Science. The author would like to thank the Institute for providing an excellent working environment. This research was supported by the ERC, StG grant number 637912 and JSPS KAKENHI, grant number 19F19019.

\section{Notations and Preliminaries}

In this section, we set up notations for use in the sequel.

\subsection{Root system}\label{sec:root systems}

Let $\Phi$ be a reduced root system, $\Phi^+$ and $\Phi^-$ a choice of positive and negative roots respectively, and $\Delta=\{\al_1,\cdots,\al_r\}$ the set of simple roots. We denote by $\Phi^\vee$ the set of coroots and $\al\mapsto \al^\vee$ the bijection between $\Phi$ and $\Phi^\vee$. We use
\[
\la \cdot,\cdot \ra:\Phi\times\Phi^\vee\to \bz
\]
to denote the canonical pairing between $\Phi$ and $\Phi^\vee$.

Let $\omega_1,\cdots,\omega_r$ be the fundamental dominant weights, which satisfy
\[
\la \omega_i,\al_j^\vee\ra=\delta_{ij} \ (\delta_{ij}= \text{Kronecker delta}).
\]
Let $\Lambda$ be the weight lattice, generated by the $\omega_i$.
Let
\[
\rho=\sum_{i=1}^r \omega_i=\dfrac{1}{2}\sum_{\al\in\Phi^+}\al.
\]
Given a weight $\lam$, define
\begin{equation}\label{eq:d lambda}
d_\lam(\al)=\la \lam+\rho,\al^\vee\ra.
\end{equation}

Let $W=W(\Phi)$ be the Weyl group of the root system $\Phi$. It is generated by the simple reflections $s_{\al_i}=s_i$ for $i=1,\cdots, r$. (We also use $s_i$ as complex variables later. But this will not cause any confusion.) The action of $s_i$ on $\Lambda$ is given by
\[
\lam\mapsto s_i(\lam):= \lam-\la \lam,\al_i^\vee\ra\al_i.
\]

Let $\ell(\cdot)$ denote the length function on $W$. Let $w_0$ denote the element of the longest length in $W$ and denote that length by $N$. For $w\in W$, define $\Phi(w)=\{\al\in\Phi^+:w(\al)\in\Phi^-\}$.

We also need the following results.

\begin{Lem}[\cite{bump-lie} Proposition 20.10] \label{lem:bump lemma}
If $w=s_{i_1}\cdots s_{i_N}$ is a reduced decomposition of $w$, then
\[
\Phi(w)=\{\al_{i_N}, \ s_{i_N}(\al_{i_{N-1}}), \ \cdots, \  s_{i_N}s_{i_{N-1}}\cdots s_{i_2}(\al_{i_1})\}.
\]
\end{Lem}

\begin{Lem}\label{lem:springer lemma}
For any $\al,\beta\in\Phi$,
\[
(s_\al(\be))^\vee=s_{\al^\vee}(\beta^\vee).
\]
\end{Lem}

\begin{proof}
This is proved in the last two lines of \cite{springer-1979} Page 4.
\end{proof}

When $\al=\al_i$ is a simple root, we write $s_i(\be^\vee)=s_{\al_i^\vee}(\be^\vee)$. The meaning of $s_i$ (either $s_{\al_i}$ or $s_{\al_i^\vee}$) is always clear in the context.

\subsection{Hilbert symbols}

Let $n>1$ be an integer and let $F$ be a number field containing the $n$-th roots of unity. Let $S$ be a finite set of places of $F$ such that $S$ contains all Archimedean places, all places ramified over $\bq$, and is sufficiently large that the ring of $S$-integers $\sco_S$ is a principal ideal domain. Embed $\sco_S$ in $F_S=\prod_{\nu\in S}F_\nu$ diagonally.

The product of local Hilbert symbols gives rise to a pairing
\[
(\ , \ )_S: F_S^\times \times F_S^\times  \to \mu_n, \qquad (a,b)_S=\prod_{\nu\in S}(a,b)_v.
\]
Here, we take the symbol to be the inverse of the symbol defined in \cite{Neukirch}; see also \cite{brubaker-bump-kubota} Section 2.
A subgroup $\Omega$ of $F_S^\times$ is called \textit{isotropic} if $(a,b)_S=1$ for all $a,b\in\Omega$.

\subsection{Gauss sum}
If $a\in\sco_S$ and $\fb$ is an ideal of $\sco_S$, let $(\frac{a}{\fb})$ be the $n$th order power residue symbol as defined in \cite{brubaker-bump-kubota}. (This depends on $S$, but we suppress this dependence from the notation.) Let $\psi$ be a nontrivial additive character $\psi$ of $F_S$ such that $\psi(x\sco_S)=1$ if and only if $x\in\sco_S$. If $a,c\in\sco_S$ and $c\neq 0$, and if $t$ is a positive integer, the Gauss sum is defined by
\[
g_t(a,c)=\sum_{d \mod c}
\left(\dfrac{d}{c\sco_S}\right)^t
\psi\left(\dfrac{ad}{c}\right).
\]

\section{Braidless Weights and Minimal Representatives}\label{sec:braidless and minimal}

In this section, we describe a parametrization of elements in the Weyl group for certain root systems. The only cases that appear in the application are given in Section \ref{sec:examples}. The reader can safely skip the materials before Section \ref{sec:examples}.

\subsection{Braidless weights}
Recall that a fundamental weight $\omega$ is \textit{braidless} if for any $\tau \in W$ the following holds (\cite{Littelmann98} Section 3):
\begin{equation}\label{eq:def of braidless}
\text{if }\al,\gamma\in\Delta \text{ are such that } \la \tau(\omega),\al^\vee \ra >0, \la \tau(\omega),\gamma^\vee \ra >0, \text{ then }\la \gamma,\al^\vee \ra =0.
\end{equation}
We call a simple root $\al$ braidless if the corresponding fundamental weight $\omega_\al$ is so.

The set of braidless weights is determined by \cite{Littelmann98} Lemma 3.1. Here we use the enumeration of fundamental weights as in \cite{Bour}.
\begin{Lem}[\cite{Littelmann98} Lemma 3.1]
A fundamental weight $\omega$ of a simple Lie algebra $\fg$ is braidless if and only if $\omega$ is a minuscule weight of $\fg$, or $\omega=\omega_1$ for $\fg$ of type $B_n$, or $\omega=\omega_n$ for $\fg$ of type $C_n$, or $\omega$ is arbitrary and $\fg$ is of rank at most two.
\end{Lem}

It is not hard to find all braidless weights (see also \cite{BF-GAFA} Section 5.2). A fundamental weight $\omega$ is braidless if and only if it is in the following collection:
\begin{itemize}
\item $\Phi$ is of type $A$,
\item $\Phi$ is of type $B$ or $C$ and $\omega=\omega_1$ or $\omega_r$,
\item $\Phi$ is of type $D$ and $\omega=\omega_1,\omega_{r-1}$ or $\omega_r$,
\item $\Phi=E_6$ and $\omega=\omega_1$ or $\omega_6$, or $\Phi=E_7$ and $\omega=\omega_7$,
\item $\Phi=G_2$.
\end{itemize}


Suppose $\omega=\omega_i$ is braidless. 
Let $W_\omega$ be the subgroup of $W$ generated by $\{s_\al:\al\in\Delta-\{\al_i\}\}$.  Denote by ${}^\omega W$ the minimal representatives of the set of right cosets $\{W_\omega \tau : \tau \in W\}$.

\begin{Lem}[\cite{Littelmann98} Lemma 3.2]\label{lem:littelmann lemma}
Let $\omega$ be braidless. If $\tau \in {}^\omega W$, then a reduced decomposition of $\tau$ is unique up to the exchange of orthogonal simple reflections.
\end{Lem}

\subsection{Parametrizations}
We describe a canonical parametrization for elements in ${}^\omega W$ if $\omega$ is braidless. Our main result is Theorem \ref{thm:main}. We also give another parametrization in terms of graphs (Corollary \ref{cor:main}). This is motivated by the Weyl group multiple Dirichlet series. (See \cite{CG-littelmann} Section 3; also compare their description and Section \ref{sec:type d} below).

Let $\tau_\ell\in {}^\omega W$ be the longest element and fix a reduced decomposition $\tau_\ell=s_{i_1}\cdots s_{i_N}$. Notice that any series of exchange of orthogonal simple reflections induces a permutation of $\{1,\cdots, N\}$. Let $S_\omega$ be the set of such permutations. Any other reduced decomposition of $\tau_\ell$ is of the form $s_{i_{\sigma(1)}}\cdots s_{i_{\sigma(N)}}$ for a unique permutation $\sigma\in S_{\omega}$.   

\begin{Lem}\label{lem:observations about permutation}
We have the following.
\begin{enumerate}
\item If $j<k$ and $s_{i_j}$ and $s_{i_k}$ do not commute, then $\sigma^{-1}(j)<\sigma^{-1}(k)$ for any $\sigma\in S_\omega$.
\item If $s_{i_j}=s_{i_k}$ and $j<k$, then $\sigma^{-1}(j)<\sigma^{-1}(k)$ for any $\sigma\in S_\omega$.
\item For any $\sigma \in S_\omega$, $\sigma(1)=1$.
\end{enumerate}
\end{Lem}

\begin{proof}
Note that $\sigma^{-1}(j)$ indicates the position of $s_{i_j}$ in new reduced expression $s_{i_{\sigma(1)}}\cdots s_{i_{\sigma(N)}}$. Part (1) is simply a consequence of Lemma \ref{lem:littelmann lemma}. (Note that $\sigma(j)<\sigma(k)$ may not be true for some $\sigma\in S_\omega$.)

We now prove part (2). In this case, there is a $j'$ such that $j<j'<k$ such that $s_{i_j}=s_{i_k}$ does not commute with $s_{i_{j'}}$. Thus by part (1) we have $\sigma^{-1}(j)<\sigma^{-1}(j')<\sigma^{-1}(k)$.

Finally, we have $s_{i_1}=s_{i_{\sigma(1)}}=s_i$ since this is the first element in the reduced decomposition of a minimal representative. If $1<\sigma(1)$ for some $\sigma \in S_\omega$, then by part (2) $\sigma^{-1}(1)<\sigma^{-1}(\sigma(1))=1$, which leads to a contradiction. This proves part (3).
\end{proof}


Define $\mathcal{C}_\omega$ to be a set of certain subsets of $\{1,\cdots, N\}$ as follows. The set $C\in\mathcal{C}_\omega$ if and only if it is of the form $\{\sigma(1),\cdots,\sigma(k)\}$ for some $ k\geq 0$ and $\sigma\in S_\omega$. (If $k=0$, we set $C=\emptyset$.)
Define a map
\[
f_\omega: \mathcal{C}_\omega\to W,\qquad C=\{\sigma(1),\cdots,\sigma(k)\}\mapsto \tau_C= s_{i_{\sigma(1)}}\cdots s_{i_{\sigma(k)}}.
\]
Clearly, this map does not depend on the choice of $\sigma$. If we start with another reduced decomposition of $\tau_\ell$, then the map $f_\omega$ changes by a conjugation induced by some $\sigma\in S_\omega$.

\begin{Thm}\label{thm:main} \
\begin{enumerate}[\normalfont(1)]
\item For $C\in\mathcal{C}_\omega$, we have $\tau_C\in {}^{\omega}W$, and $s_{i_{\sigma(1)}}\cdots s_{i_{\sigma(k)}}$ is a reduced expression.
\item The map $f_\omega: C_\omega \to {}^{\omega}W$ is a bijection.
\end{enumerate}
\end{Thm}

\begin{Ex}\label{ex:first one}
Let $\fg=\fso_{8}$ and we choose the following enumeration:
\[
 \begin{tikzpicture}
        \node (a1) at (-.707,0.707) {$\al_1$};
        \node (a2) at (-.707,-0.707) {$\al_2$};
        \node (a3) at (0,0) {$\al_3$};
        \node (a5) at (1,0) {$\al_4$};
  \draw [-] (a1) to (a3);
  \draw [-] (a2) to (a3);
  \draw [-] (a3) to (a5);
  \end{tikzpicture}
\]
This is different from the usual enumeration. (It is a good enumeration in the sense of Section \ref{sec:good enumeration}.)
Let  $\omega=\omega_4$. The longest element in ${}^\omega W$ is $\tau_\ell=s_4s_3s_1s_2s_3s_4$. Then
$S_{\omega}=\{\mathrm{id}, (34)\}$,
\[
\mathcal{C}_\omega=\{
\emptyset, \ \{1\}, \ \{1,2\}, \ \{1,2,3\}, \ \{1,2,4\}, \ \{1,2,3,4\}, \ \{1,2,3,4,5\}, \ \{1,2,3,4,5,6\}
\},
\]
and
\[
{}^\omega W=\{ I,\ s_4,\ s_4s_3, \ s_4s_3 s_1,\ s_4s_3s_2, \ s_4s_3s_1s_2,\ s_4s_3s_1s_2s_3,\ s_4s_3s_1s_2s_3s_4\}.
\]

\end{Ex}

\begin{proof}
(1)\ Clearly $s_{i_{\sigma(1)}}\cdots s_{i_{\sigma(k)}}$ is a reduced expression. Otherwise, $s_{i_{\sigma(1)}}\cdots s_{i_{\sigma(N)}}$ is not reduced. 
To show that $\tau_C$ is a minimal representative, we need to show that $\Phi(\tau_C^{-1})\cap (\Delta-\{\al_i\})=\emptyset$. (See \cite{BT72} Proposition 3.9 or \cite{casselman1995} Lemma 1.1.2. The statement we need is slightly different from these two references but can be deduced from them.)  Without loss of generality, we may assume $\sigma$ is the identity element.
By Lemma \ref{lem:bump lemma},
\[
\Phi(\tau_\ell^{-1})=\{\al_{i_1},\ s_{i_1}(\al_{i_2}),\ \cdots,\ s_{i_1}\cdots s_{i_{N-1}}(\al_{i_N}) \}
\]
 and \[
 \Phi(\tau_C^{-1})=\{\al_{i_1},\ s_{i_1}(\al_{i_2}),\ \cdots,\ s_{i_1}\cdots s_{i_{k-1}}(\al_{i_k}) \}.
\]
Thus $\Phi(\tau_C^{-1})\subset \Phi(\tau_\ell^{-1})$. The fact that $\tau_\ell$ is minimal implies that $\tau_C$ is minimal as well.

(2)\ We first show that this map is a surjection. Let $\tau\in {}^\omega W$. Let $\tau^\ell$ be the longest element in $W_\omega$. Then $w_0=\tau^\ell\tau_\ell$ is the longest element in $W$, and $\ell(w_0)-\ell(\tau^{-1}_\ell\tau)=\ell(w_0\cdot \tau_\ell^{-1}\tau)$.  Thus
\[
\ell(w_0)-\ell(\tau^{-1}\tau_\ell)=\ell(\tau^\ell\tau).
\]
By \cite{casselman1995} Lemma 1.1.2, we know $\ell(\tau^\ell\tau)=\ell(\tau^\ell)+\ell(\tau)$ and $\ell(w_0)=\ell(\tau^\ell)+\ell(\tau_\ell)$.  Thus
\[
\ell(\tau)+\ell(\tau^{-1}\tau_\ell)=\ell(\tau_{\ell}).
 \]
 This means that if we choose reduced decompositions $s_{j_1}\cdots s_{j_k}$ and $s_{j_{k+1}}\cdots s_N$ for $\tau$ and $\tau^{-1}\tau_\ell$ respectively, then $s_{j_1}\cdots s_{j_N}$ is a reduced decomposition for $\tau_\ell$. Thus there is a permutation $\sigma\in S_\omega$ such that $j_i=\sigma(i)$ and $\tau=s_{i_{\sigma(1)}}\cdots s_{i_{\sigma(k)}}$.

Next we show $f_\omega$ is injective.  Without loss of generality, we show that if there are two sets
\[
\{1,\cdots, k\} ,\qquad \{\sigma(1),\cdots,\sigma(k)\}
\]
such that $s_{i_1}\cdots s_{i_k}=s_{i_{\sigma(1)}}\cdots s_{i_{\sigma(k)}}$, then $\{1,\cdots, k\}=\{\sigma(1),\cdots, \sigma(k)\}$. By part (1) $s_{i_1} \cdots s_{i_k}=s_{i_{\sigma(1)}} \cdots s_{i_{\sigma(k)}}$ are minimal length representatives of a right coset. By Lemma \ref{lem:littelmann lemma}, these words differ by a permutation of commuting simple reflections. Therefore the multi-sets $\{i_1,\cdots, i_k\}$ and $\{i_{\sigma(1)},\cdots, i_{\sigma(k)}\}$ are the same.

We now show that $\sigma^{-1}$ maps the set $\{1,\cdots,k\}$ to itself. Consider $1\leq t \leq r$ and let $1\leq j_1  < \cdots < j_m \leq N$ denote the instances where $s_t$ appears in $s_{i_1}\cdots s_{i_N}$. Then by Lemma \ref{lem:observations about permutation} part (2) the positions of $s_t$ in $s_{i_{\sigma(1)}} \cdots s_{i_{\sigma(N)}}$ are $\sigma^{-1}(j_1) < \cdots < \sigma^{-1}(j_m)$.
If $s_t$ appears $m'\leq m$ times in $s_{i_1}\cdots s_{i_k}$, then by the equality of multisets it appears $m'\leq m$ times in $s_{i_{\sigma(1)}} \cdots s_{i_{\sigma(k)}}$ as well. Thus exactly $m'$ of $\sigma^{-1}(j_1) < \cdots < \sigma^{-1}(j_m)$ fall into $[1,k]$.

So if $j_{m'}\leq k<j_{m'+1}$ (or $m=m'\leq k$), then $\sigma^{-1}(j_{m'})\leq k < \sigma^{-1}(j_{m'+1})$ (or $\sigma^{-1}(j_{m'})=\sigma^{-1}(j_m)\leq k$). Therefore the subset of $\{1,\cdots,k\}$ corresponding to the appearances of $s_t$ is mapped to a subset of $\{1,\cdots,k\}$ under $\sigma^{-1}$. This is true for any $1\leq t \leq r$, so the proof is complete.


\end{proof}

We now introduce another description which is slightly easier to work with. Our motivation comes from the decoration graph in  \cite{CG-littelmann} Section 3. We now explain it carefully.  We again start with a reduced decomposition $\tau_\ell=s_{i_1}\cdots s_{i_N}$ and define a directed graph $T_{i_1,\cdots,i_N}$ (with labelling) as follows.  The vertices of $T_{i_1,\cdots, i_N}$ are $\{1,\cdots, N\}$. An edge $j\to k$ is in  $T_{i_1,\cdots, i_N}$ if and only if $s_{i_j}$ and $s_{i_k}$ do not commute and there exists $\sigma\in S_\omega$ and $1\leq l\leq N$ such that $\sigma(l)=j$ and $\sigma(l+1)=k$. (In other words, $s_{i_j}$ and $s_{i_k}$ can be moved by a series of orthogonal permutations so that they are next to each other). We also label the vertex $j$ by $s_{i_j}$. It is easy to see that the graph $T_{i_1,\cdots, i_N}$ is connected. Also notice  that if $(j\to k)\in T_{i_1,\cdots, i_N}$, then $\sigma^{-1}(j)<\sigma^{-1}(k)$ for all $\sigma\in S_\omega$.

\begin{Ex}\label{ex:2.3}
Let $\fg=\fso_{8}$ and $\omega=\omega_4$ (the enumeration used here is the same as Example \ref{ex:first one}). The longest element is $\tau_\ell=s_4s_3s_1s_2s_3s_4$. Then the graph $T_{4,3,1,2,3,4}$ is
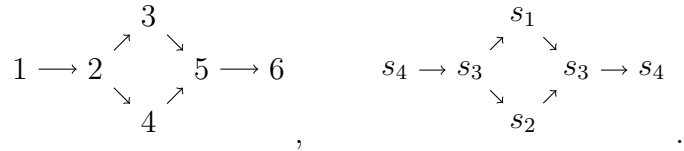
\[
  \begin{tikzpicture}
        \node (A) at (-1.707,0) {$1$};
        \node (B) at (-0.707,0) {$2$};
        \node (C) at (0,0.707) {$3$};
        \node (D) at (0,-0.707) {$4$};
        \node (E) at (0.707,0) {$5$};
        \node (F) at (1.707,0) {$6$};

                        \draw [->] (A) to (B);
  \draw [->] (B) to (C);
  \draw [->] (B) to (D);
  \draw [->] (C) to (E);
  \draw [->] (D) to (E);
  \draw [->] (E) to (F);
                        \end{tikzpicture},\qquad
        \begin{tikzpicture}
        \node (A) at (-1.707,0) {$s_4$};
        \node (B) at (-0.707,0) {$s_3$};
        \node (C) at (0,0.707) {$s_1$};
        \node (D) at (0,-0.707) {$s_2$};
        \node (E) at (0.707,0) {$s_3$};
        \node (F) at (1.707,0) {$s_4$};

                        \draw [->] (A) to (B);
  \draw [->] (B) to (C);
  \draw [->] (B) to (D);
  \draw [->] (C) to (E);
  \draw [->] (D) to (E);
  \draw [->] (E) to (F);
                        \end{tikzpicture}.
\]
\end{Ex}

 Let $\mathcal{G}_\omega$ be set of subgraphs of $T_{i_1,\cdots,i_N}$ with the following properties:
\begin{enumerate}
\item If $T$ is not the empty subgraph of $T_{i_1,\cdots,i_N}$, then $1\in T$.
\item If $k\in T$, and the edge $(j\to k)\in T_{i_1,\cdots,i_N}$, then $j\in T$.
\end{enumerate}

\begin{Ex}[Continuation of Example \ref{ex:2.3}]
The set $\mathcal{G}_\omega$ consists of the following subgraphs:
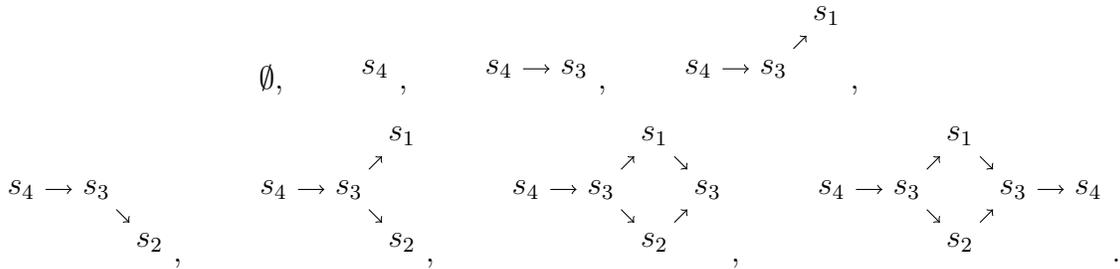
\[
\emptyset, \qquad
 \begin{tikzpicture}
        \node (A) at (-1.707,0) {$s_4$};
                        \end{tikzpicture},
                        \qquad
        \begin{tikzpicture}
        \node (A) at (-1.707,0) {$s_4$};
        \node (B) at (-0.707,0) {$s_3$};

                        \draw [->] (A) to (B);
                        \end{tikzpicture},
                        \qquad
        \begin{tikzpicture}
        \node (A) at (-1.707,0) {$s_4$};
        \node (B) at (-0.707,0) {$s_3$};
        \node (C) at (0,0.707) {$s_1$};

                        \draw [->] (A) to (B);
  \draw [->] (B) to (C);
                        \end{tikzpicture},
\]
\[
        \begin{tikzpicture}
        \node (A) at (-1.707,0) {$s_4$};
        \node (B) at (-0.707,0) {$s_3$};

        \node (D) at (0,-0.707) {$s_2$};

                        \draw [->] (A) to (B);

  \draw [->] (B) to (D);

                        \end{tikzpicture},
                        \qquad
        \begin{tikzpicture}
        \node (A) at (-1.707,0) {$s_4$};
        \node (B) at (-0.707,0) {$s_3$};
        \node (C) at (0,0.707) {$s_1$};
        \node (D) at (0,-0.707) {$s_2$};

                        \draw [->] (A) to (B);
  \draw [->] (B) to (C);
  \draw [->] (B) to (D);

                        \end{tikzpicture},
                     \qquad
        \begin{tikzpicture}
        \node (A) at (-1.707,0) {$s_4$};
        \node (B) at (-0.707,0) {$s_3$};
        \node (C) at (0,0.707) {$s_1$};
        \node (D) at (0,-0.707) {$s_2$};
        \node (E) at (0.707,0) {$s_3$};

                        \draw [->] (A) to (B);
  \draw [->] (B) to (C);
  \draw [->] (B) to (D);
  \draw [->] (C) to (E);
  \draw [->] (D) to (E);
                        \end{tikzpicture},
                        \qquad
                        \begin{tikzpicture}
        \node (A) at (-1.707,0) {$s_4$};
        \node (B) at (-0.707,0) {$s_3$};
        \node (C) at (0,0.707) {$s_1$};
        \node (D) at (0,-0.707) {$s_2$};
        \node (E) at (0.707,0) {$s_3$};
        \node (F) at (1.707,0) {$s_4$};

                        \draw [->] (A) to (B);
  \draw [->] (B) to (C);
  \draw [->] (B) to (D);
  \draw [->] (C) to (E);
  \draw [->] (D) to (E);
  \draw [->] (E) to (F);
                        \end{tikzpicture}.
\]
\end{Ex}

\begin{Lem}
There is a natural bijection between $\mathcal{G}_\omega$ and $\mathcal{C}_\omega$.
\end{Lem}
\begin{proof}
The map from $\mathcal{G}_\omega$ to $\mathcal{C}_\omega$ is clear. Let $T\in \mathcal{G}_\omega$ and  $\{j_1,\cdots,j_k\}$ be the vertices of $T$. We simply send a graph $T$ to its set of vertices. To verify that this set satisfies the desired properties, we argue by induction on $k$.  If  $k=0$ or $1$, this is clear.
Now suppose $k\geq 2$. Recall that if $(j\to j')\in T$, then $j<j'$. Thus there must be a maximal element (say, $j_k$) in $T$ (i.e. there is no edge coming out from this maximal element). Delete $j_k$ and we obtain a  graph in $\mathcal{G}_\omega$ with shorter length. Now apply induction, and (possibly) rearrange the indices in $\{j_1,\cdots, j_{k-1}\}$, we see that there is a reduced decomposition of $\tau_\ell$ which is of the form
\[
(s_{i_{j_1}}\cdots s_{i_{j_{k-1}}}) \cdots s_{i_{j_k}} \cdots
\]
We need to show that  $s_{i_{j_k}}$ can be moved (by a series of orthogonal permutations) to the position which is right after $s_{i_{j_{k-1}}}$. If this is not true, then there is a simple reflection $s_{i_{j'}}$ strictly between $s_{i_{j_{k-1}}}$ and $s_{i_{j_{k}}}$ such that $s_{i_{j'}}$ does not commute with $s_{i_{j_{k}}}$. Choose the maximal $j'$. Then $(j'\to j_k)\in T_{i_1,\cdots,i_N}$ and thus $j'\in T$. This is a contradiction.

Conversely, given a nonempty subset $\{\sigma(1),\cdots,\sigma(k)\}$ for some $k$ and $\sigma\in S_\omega$, we can define a subgraph $T$ by taking all the possible edges between these vertices.
First of all, by Lemma \ref{lem:observations about permutation}, $\sigma(1)=1\in T$.
If there is $(j'\to \sigma(j))\in T_{i_1,\cdots, i_N}$ for some $j'$ (might not be in $T$) and $1\leq j\leq k$. Then $\sigma^{-1}(j')<j$ and therefore $j'=\sigma(\sigma^{-1}(j'))\in T$. This implies that $T\in \mathcal{G}_\omega $ and  we are done.
\end{proof}
We have the following immediate corollary.

\begin{Cor}\label{cor:main}
There is a natural bijection between ${}^\omega W$ and $\mathcal{G}_\omega$.
\end{Cor}

\begin{Rem}\label{rem:graph}
If we start with a different reduced decomposition of $\tau_\ell$, then the graph is obtained by relabelling the vertices according to a permutation $\sigma\in S_\omega$. In Appendix \ref{app:examples}, we work out all the examples. By abuse of notation, we simply call it $T_{\tau_\ell}$. We only give the labelling without giving the indices.
\end{Rem}

\subsection{Good enumeration}\label{sec:good enumeration}
We now recall the notion of \textit{good enumeration} in the sense of \cite{Littelmann98} Section 4.
Let $D$ be the Dynkin diagram of $\fg$.  We write $D-\{\al\}$ for the diagram obtained by removing the node of $\al$ from $D$.

Suppose that the set of simple roots of $\fg$ admits an enumeration
\[
\Delta=\{\al_1,\cdots, \al_n\} \text{ such that }\al_i \text{ is braidless for }D-\{\al_{i+1},\cdots,\al_{n}\}
\]
for all $i=1,\cdots, n$. We call this  a \textit{good enumeration}. (Note that there is a misprint in \cite{Littelmann98}, before Remark 4.1.)

Note that all simple Lie algebras admit such a \textit{good} enumeration except the ones  of type $F_4$ and $E_8$. For such good enumeration of the respective types, see \cite{Littelmann98} Section 5-8.

Let $\fg$ be a simple algebra except the ones  of type $F_4$ and $E_8$ and take a good enumeration of the set of simple roots.
Let $W_j$ be the Weyl group generated by the simple reflections $s_{1},\cdots, s_{j}$ and let ${}^j W_j$ be the set of minimal representatives in $W_j$ of $W_{j-1}\bs W_j$. Let ${}^j\mathcal{G}_j$ be $\mathcal{G}_\omega$ where the root system is generated by $\{\al_1,\cdots,\al_j\}$ and $\omega=\omega_j$.

For a simple Lie algebra admitting a good enumeration, we obtain a natural parametrization of the Weyl group $W(\fg)$.

\begin{Prop}\label{prop:main result in weyl}
There is a natural bijection
\[
{}^1\mathcal{G}_1\times \cdots \times {}^n\mathcal{G}_n
\simeq {}^1W_1\times \cdots \times {}^nW_n\simeq W(\fg).
\]
\end{Prop}

\begin{proof}
The first bijection is an immediate consequence of Corollary \ref{cor:main} and the definition of good enumerations. The second bijection is given by
\[
{}^1W_1\times \cdots \times {}^nW_n\simeq W(\fg),\qquad (w_1,\cdots, w_n)\mapsto w_1 \cdots w_n.
\]
\end{proof}

The reduced decomposition of the longest element $w_0$ obtained in this way is called \textit{nice} decomposition.

\subsection{Examples}\label{sec:examples}
One can write down a more explicit version of Proposition \ref{prop:main result in weyl} using the description in Appendix \ref{app:examples}.
To avoid interrupting the application, here we only give two examples (type $A$ and type $C$) which are required in Section \ref{sec:WGMDS}.

We choose the following good enumerations:
\begin{equation}\label{eq:good enum type A C}
\begin{aligned}
\text{Type }A_r: & \  \begin{tikzpicture}
        \node (a1) at (-2,0) {$\al_1$};
        \node (a2) at (-1,0) {$\al_2$};
        \node (a3) at (0,0) {$\al_3$};
        \node (a4) at (1,0) {$\cdots$};
        \node (a5) at (2,0) {$\al_r$};
  \draw [-] (a1) to (a2);
  \draw [-] (a2) to (a3);
  \draw [-] (a3) to (a4);
  \draw [-] (a4) to (a5);
                        \end{tikzpicture},\\
\text{Type } C_r: & \ \begin{tikzpicture}
        \node (a1) at (-2,0) {$\al_1$};
        \node (a2) at (-1,0) {$\al_2$};
        \node (a3) at (0,0) {$\al_3$};
        \node (a4) at (1,0) {$\cdots$};
        \node (a5) at (2,0) {$\al_r$};
\draw [-] (-1.65,0.04) to (-1.35,0.04);
\draw [-] (-1.65,-0.04) to (-1.35,-0.04);
  \draw [-] (a2) to (a3);
  \draw [-] (a3) to (a4);
  \draw [-] (a4) to (a5);
                        \end{tikzpicture}.
                        \end{aligned}
\end{equation}

Notice that in both cases, the reduced decomposition for the longest element $\tau_\ell$ in ${}^\omega W$ is unique when $\omega=\omega_r$. Thus the corresponding parametrization for ${}^\omega W$ for these two types is indeed simpler.

In type $A$, given $0\leq a_i\leq i$, define
\[
\pi_{a_i}=
\begin{cases}
  s_i\cdots s_{i-a_i+1}, & \mbox{if } 1\leq a_i\leq i, \\
  e \ \text{(the identity element)}, & \mbox{if }a_i=0.
\end{cases}
\]
In type $C$, given $0\leq a_i\leq 2i-1$, define
\[
\pi_{a_i}=
\begin{cases}
  s_i\cdots s_{i-a_i+1}, & \mbox{if } 1\leq a_i\leq i, \\
   s_i\cdots s_2 s_1 s_2 \cdots s_{a_i-i+1}, & \mbox{if } i+1\leq a_i\leq 2i-1, \\
  e \ \text{(the identity element)}, & \mbox{if }a_i=0.
\end{cases}
\]
Define
\[
\mathcal{D}_r=
\begin{cases}
\text{Type }A:& \ \{(a_1,\cdots,a_r):0\leq a_i\leq i, \text{ for all }i\},\\
\text{Type }C :& \ \{(a_1,\cdots,a_r):0\leq a_i\leq 2i-1, \text{ for all }i\}.
\end{cases}
\]

We now state the following explicit parametrization for the Weyl group in the cases of type $A$ and type $C$. See also \cite{Cai18} for the case of type $A$.

\begin{Prop}\label{prop:bijection type A C}
The nice decompositions obtained for the enumerations in \eqref{eq:good enum type A C} are
\[
\begin{aligned}
\text{Type }A : & \ w_0=(s_1)(s_2s_1)\cdots (s_r\cdots s_1),\\
\text{Type }C : & \ w_0=(s_1)(s_2s_1s_2)\cdots (s_r s_{r-1}\cdots s_1\cdots s_{r-1} s_r).
\end{aligned}
\]
The map
\[
\mathcal{D}_r\to W, \qquad (a_1,\cdots,a_r)\mapsto \pi_{a_1}\cdots \pi_{a_r}
\]
is a bijection.
\end{Prop}

\begin{proof}
By the results in Sections \ref{sec:example type A} and \ref{sec:example type BC},
\[
\begin{aligned}
\text{Type }A :{}^i W_i =  &\{\pi_{a_i}:0\leq a_i\leq i\}, \\
\text{Type }C :{}^i W_i = &\{\pi_{a_i}:0\leq a_i\leq 2i-1\}  .
\end{aligned}
\]
The result now follows from Proposition \ref{prop:main result in weyl}.
\end{proof}

\begin{Ex}
To make a connection with the applications later (see Proposition \ref{prop:bijection}), we now include a small rank example. For type $C$, the set $\mathcal{D}_2$ is
\[
\{
(0,0), \ (0,1),\ (0,2),\ (0,3),\ (1,0),\ (1,1),\ (1,2),\ (1,3)
\}.
\]
We now write the simple reflections in $\pi_{a_1}\pi_{a_2}$ in an array so that the $i$-th row from the bottom consists of the simple reflections appearing in $\pi_{a_i}$:
\[
e, \
\begin{pmatrix} s_2 & \ & \  \\ & & \end{pmatrix}, \
\begin{pmatrix} s_2 & s_1 & \ \\ & &  \end{pmatrix}, \
\begin{pmatrix} s_2 & s_1 & s_2 \\ & &  \end{pmatrix},
\]
\[
\begin{pmatrix} \  & \ & \ \\ & s_1 & \end{pmatrix}, \
\begin{pmatrix} s_2 & \  & \ \\ & s_1 & \end{pmatrix}, \
\begin{pmatrix} s_2 & s_1 & \ \\ & s_1 & \end{pmatrix},\
\begin{pmatrix} s_2 & s_1 & s_2 \\ & s_1 & \end{pmatrix}.
\]
The reader can compare this with the nonzero entries in the patterns corresponding to the orbit of the highest weight vector in \cite{BBF-typec-crystal} Figure 2.1. (These patterns are stable patterns in the sense of Definition \ref{def:stable patterns}.)

\end{Ex}

\section{Berenstein-Zelevinsky-Littelmann Patterns and Crystal Graphs}\label{sec:BZL patterns}


\subsection{BZL patterns}
Given a semisimple algebraic group $G$ of rank $r$ and a simple $G$-module $V_\lam$ of highest weight $\lam$, we may associate a crystal graph $\scb_\lam$ to $V_\lam$. That is, there exists a corresponding simple module for the quantum group $U_q(\mathrm{Lie}(G))$ having the associated crystal graph structure. The crystal graph encodes data from the representation $V_\lam$, and can be regarded as a kind of ``enhanced character'' for the representation. We refer the reader to \cite{HK2002} Chapter 4 for the background materials on crystal graphs.

The vertices of $\scb_\lam$ are in bijection with a basis of weight vectors for the highest weight representation $V_\lam$. Given a vertex $v\in\scb_\lam$, let $\mathrm{wt}(v)$ be the weight of the corresponding weight vector. The edges of $\scb_\lambda$ are ``colored'' by $i$ for simple roots $\al_i$. Two vertices $v_1,v_2$ are connected by a (directed) edge from $v_1$ to $v_2$ of color $i$ if the Kashiwara raising operator $e_{i}:=e_{\al_i}$ takes $v_1$ to $v_2$. In this case, $\mathrm{wt}(v_2)=\mathrm{wt}(v_1)+\al_i$. If the vertex $v$ has no outgoing edge of color $i$, we set $e_{i}(v)=0$. The Kashiwara lowering operator is denoted by $f_i$.

\begin{Rem}\label{rem:different description}
Notice that we follow the formulation in \cite{Littelmann98}. This will make the comparison in Section \ref{sec:WGMDS} more natural. The formulation in \cite{BBF-crystal} takes the opposite direction. However, as explained in \cite{BBF-crystal} Proposition 1, they can be tied by the use of the Sch\"utzenberger involution.
\end{Rem}

Littelmann \cite{Littelmann98} gives a combinatorial model for the crystal graph as follows. Fix a reduced decomposition of the longest element $w_0$ of the Weyl group of $G$ into simple reflections:
\[
w_0=s_{i_1}s_{i_2}\cdots s_{i_N}.
\]
Given an element $v$ (i.e. vertex) of the crystal $\scb_\lam$, let $b_1$ be the maximal integer such that $v_1:=e_{{i_1}}^{b_1}(v)\neq 0$. Similarly, let $b_2$ be the maximal integer such that $v_2:=e_{{i_2}}^{b_2}(v_1)\neq 0$. Continuing in this fashion in order of the simple reflections appearing in $w_0$, we obtain a string of non-negative integers
\[
\BZL(v):=(b_1,\cdots, b_N).
\]
This is the so-called Berenstein-Zelevinsky-Littelmann pattern (or BZL-pattern for short). By well-known properties of the crystal graph, we are guaranteed that for any reduced decomposition of the longest element and an arbitrary element $v$ of the crystal, the path $e_{{i_N}}^{b_N}\cdots e_{{i_1}}^{b_1}(v)$ through the crystal always terminates at $v_{\lam}$, corresponding to the highest weight vector of the crystal graph $\scb_\lam$.

Littelmann (\cite{Littelmann98} Proposition 1.5) proves that, for any fixed reduced decomposition, the set of all sequences $(b_1,\cdots, b_N)$ as we vary over all vertices of all highest weight crystals $V_\lam$ associated to $G$ fill out the integer lattice points of a cone in $\br^N$. The inequalities describing the boundary of this cone depend on the choice of reduced decomposition. In general, it is a difficult task to give a set of inequalities.  For the nice decomposition coming from a good enumeration (see Section \ref{sec:good enumeration}), Littelmann shows that the cone is defined by a rather simple set of inequalities (\cite{Littelmann98} Section 4). We describe this for type $A$ and $C$ below, and refer this as the set of cone inequalities.

For any fixed highest weight $\lam$, the set of all sequences $(b_1,\cdots,b_N)$ for the crystal $\scb_\lam$ are integer lattice points of a polytope in $\br^N$. The boundary of the polytope consists of the hyperplanes defined by the cone inequalities independent of $\lam$, together with additional hyperplanes dictated by the choice of $\lam$ (\cite{Littelmann98} Proposition 1.5). We call the latter ones polytope inequalities.

\subsection{Inequalities}

We now describe these two sets of inequalities. We start with the set of polytope inequalities, which depends on the highest weight $\lam$. It is given by
\[
\begin{aligned}
\psi_N&: \ b_N\leq \la \lam, \al_{i_N}^\vee \ra,\\
\psi_{N-1}&:\ b_{N-1}\leq \la \lam-b_N\al_{i_N},\al_{i_{N-1}}^\vee\ra,\\
\cdots\\
\psi_1&:\ b_1\leq \la \lam-b_N\al_{i_N}-\cdots-b_2\al_{i_2},\al_{i_1}^\vee\ra.
\end{aligned}
\]

The set of cone inequalities is more involved. 
We only describe it in types $A$ and $C$ for the nice decompositions given in Proposition \ref{prop:bijection type A C}. The interested reader is referred to \cite{Littelmann98} Section 5--8 for further details.


We arrange a $\BZL$ pattern $\BZL(v)=(b_1, b_2, \cdots, b_N)$ into an array 
as follows. In type $A$, we set
\[
\BZL(v)=
\begin{pmatrix}
\cdots &\cdots & \cdots\\
&b_2&b_3\\
&&b_1\\
\end{pmatrix}=
\begin{pmatrix}
\cdots &\cdots & \cdots\\
 & c_{r-1,r-1} & c_{r-1,r} \\
 & & c_{r,r} \\
\end{pmatrix}.
\]
In type $C$, we set
\[
\BZL(v)=
\begin{pmatrix}
\cdots &\cdots & \cdots&\cdots &\cdots\\
&b_2&b_3&b_4&\\
&&b_1&&\\
\end{pmatrix}=
\begin{pmatrix}
\cdots &\cdots & \cdots&\cdots &\cdots\\
&c_{r-1, r-1}&c_{r-1,r}&\bar{c}_{r-1,r-1}&\\
&&c_{r,r}&&\\
\end{pmatrix}.
\]
In the following, we may use either form, depending on the context.

Note that the corresponding simple reflections used to define $b_i$ are arranged as
\[
\text{Type }A:
\begin{pmatrix}
\cdots &\cdots & \cdots\\
&s_2&s_1\\
&&s_1\\
\end{pmatrix},
\qquad
\text{Type }C:
\begin{pmatrix}
\cdots &\cdots & \cdots&\cdots &\cdots\\
 & s_2 & s_1 & s_2 & \\
 & & s_1 & & \\
\end{pmatrix}
\]

We now describe the set of cone inequalities. In both cases, we require the entries in each row to be non-negative and weakly decreasing.
In other words, for $1\leq i\leq r$,
\[
\begin{aligned}
\text{Type }A: & \  c_{i,i}\geq c_{i,i+1} \geq \cdots \geq c_{i,r} \geq 0,\\
\text{Type }C: & \ c_{i,i}\geq c_{i,i+1}\geq \cdots \geq c_{i,r} \geq  \bar{c}_{i,r-1} \geq  \cdots \geq \bar{c}_{i,i} \geq 0.
\end{aligned}
\]
Notice that each entry $b_i$ appears on the left-hand side of a unique inequality. (This does not hold in general.) We write this inequality as $\phi_i$. For example, in type $A$ we have $b_2=c_{r-1,r-1}$ and $\phi_2$ is $c_{r-1,r-1}\geq c_{r-1,r}. $

\subsection{Decorations}

In \cite{BBF-crystal} page 1088, a decoration rule is introduced for type $A$ patterns in the description of the Weyl group multiple Dirichlet series. A similar decoration is also given for type $C$ patterns in \cite{BBF-typec-crystal} Section 2.4.  We recall it here. Such a decoration for other types necessarily requires a modification. (Thus the definition of stable and unstable patterns below is not clear in general.)


\begin{Def}\label{def:stable and unstable}
An entry $b_i$ is circled (resp. boxed) if the inequality $\phi_i$ (resp. $\psi_i$) is an equality.
\end{Def}

In particular, a zero entry is always circled.

\subsection{Stable and unstable patterns}\label{sec:stable and unstable}

We now define stable and unstable patterns. The meaning of these two terms will become evident in Section \ref{sec:comparison}.

\begin{Def}\label{def:stable patterns}
We call a pattern \textit{stable} if every entry has one and only one decoration. Otherwise, we call it an \textit{unstable} pattern.
\end{Def}

Note that a zero entry is not boxed in a stable pattern. It is easy to see that, for a highest weight representation $V_\lam$ to have a stable pattern, we need to assume $\la\lam,\al_i^\vee\ra>0$ for all $i$. We make this assumption from now on.

We have the following observation. This result is used in Section \ref{sec:comparison}.

\begin{Lem}\label{lem:unstable}
Let $\al_0^\vee$ be the highest coroot. For any $v\in\scb_{\lam}$ and $\BZL(v)=(b_1,\cdots, b_N)$,
\[
\max_{j} b_j\leq \la \lam,\al_0^\vee \ra.
\]
If $b_j=\la\lam,\al_0^\vee\ra$, then $b_j$ is boxed.
\end{Lem}

\begin{proof}
We first show that if $\mu$ is a weight of the highest representation $V_{\lam}$, and $\al$ is a root, then $\la \mu,\al^\vee \ra \leq \la \lam,\al_0^\vee \ra$.

Let $w$ be a Weyl group element such that $w\mu$ is in the positive Weyl chamber. Then $\la \mu,\al^\vee \ra=\la w\mu,w\al^\vee \ra$. Since $\al_0^\vee - w\al^\vee$ is a non-negative linear combination of simple coroots and $w\mu$ is dominant, we have $\la w\mu, w\al^\vee\ra \leq \la w\mu,\al_0^\vee\ra$. Finally, $\lam-w\mu$ is a non-negative linear combination of simple roots, and $\al_0^\vee$ is the highest coroot, therefore $\la \lam-w\mu,\al_0^\vee \ra \geq 0$ (\cite{Humphreys78} 10.4, Lemma A). This shows that $\la \mu,\al^\vee \ra \leq \la \lam,\al_0^\vee\ra$.

From the polytope inequality $\psi_j$, we have
\begin{equation}\label{eq:polytope ineq estimation}
b_j\leq \left \la \lam-\sum_{k=j+1}^N b_k \al_k,\al_{i_j}^\vee \right \ra \leq \la \lam,\al_0^\vee\ra.
\end{equation}
The last inequality follows from the fact that $\lam-\sum_{k=j+1}^N b_k \al_k$ is in the weight diagram of $V_\lam$.

Now suppose $b_j=\la \lam,\al_0^\vee\ra$. From \eqref{eq:polytope ineq estimation} we know that $\psi_j$ becomes
\[
b_j\leq \la\lam,\al_0^\vee\ra=b_j.
\]
This implies that $b_j$ is boxed.
\end{proof}

Notice that it is still possible to obtain the maximum with the non-highest weight and a non-highest coroot.

\begin{Lem}\label{lem:unstable decoration}
If $b_{j-1}$ and $b_j$ are in the same row, $b_{j-1}$ is circled and $b_{j}$ is boxed, then the pattern is unstable.
\end{Lem}

\begin{proof}
Let $j'$ be the largest index such that $j'<j$ and $i_{j'}=i_j$. This means that $b_{j'}$ is the last entry before $b_j$ corresponding to the same simple reflection. We explain why such a $j'$ exists. Suppose we are considering the $m$-th row from the bottom. Notice that $b_j$ is not the first entry in the row.  If $i_j=m$, then we choose the first entry in this row. If $i_j<m$, then then there is an entry in a row below corresponding to the same simple reflection. Note that in both cases we have $j'<j-1$.

We have $\la \al_{i_{j-1}},\al_{i_j}^\vee\ra = -1$ unless the type is $C$ and $i_{j-1}=1$. We treat this special case first. In this case $j'=j-2$. As $b_{j-1}$ is circled and $b_{j}$ is boxed,
\begin{equation}\label{eq:property of stable I}
b_{j-1}=b_{j}= \left\la \lam-\sum_{k=j+1}^N b_k\al_{i_k},\al_{i_j}^\vee\right \ra.
\end{equation}
Then the polytope inequality $\psi_{j'}$ for $b_{j'}$ is
\[
\begin{aligned}
b_{j'}=b_{j-2} & \leq \left\la \lam-\sum_{k=j-1}^N b_k\al_{i_k},\al_{i_{j'}}^\vee \right\ra \\
=& \left\la \lam-\sum_{k=j+1}^N b_k\al_{i_k},\al_{i_j}^\vee \right\ra - b_j\la \al_{i_j}, \al_{i_j}^\vee\ra - b_{j-1} \la \al_{i_{j-1}},\al_{i_j}^\vee \ra \\
=&b_j- 2b_j -b_{j-1}\la \al_1,\al_2^\vee \ra \\
=&b_j - 2 b_j + 2b_j=b_j.
\end{aligned}
\]
Here $b_{j-2}$ and $b_j$ are in the same row. This means both $\phi_{j-2}$ and $\psi_{j-2}$ hold with equality. Therefore the pattern is unstable.

In all other cases, we have $\la \al_{i_{j-1}}, \al_{i_j}^\vee \ra =-1$ and therefore substituting \eqref{eq:property of stable I} into the inequality $\psi_{j'}$ as above implies
\begin{equation}\label{eq:property of stable}
\begin{aligned}
b_{j'}&\leq \left\la \lam-\sum_{k=j'+1}^N b_k\al_{i_k},\al_{i_{j'}}^\vee \right\ra\\
&=\left\la \lam-\sum_{k=j+1}^N b_k\al_{i_k},\al_{i_j}^\vee \right\ra - b_j \la \al_{i_j},\al_{i_j}^\vee \ra - b_{j-1} \la \al_{i_{j-1}},\al_{i_j}^\vee\ra  - \left\la\sum_{k=j'+1}^{j-2} b_k\al_{i_k},\al_{i_j}^\vee \right \ra\\
&=b_j-2b_j+b_{j-1}- \left\la \sum_{k=j'+1}^{j-2} b_k\al_{i_k},\al_{i_j}^\vee \right\ra\\
&=- \sum_{k=j'+1}^{j-2} b_k \la \al_{i_k},\al_{i_j}^\vee \ra.
\end{aligned}
\end{equation}

It follows from the choice of $j'$ and the reduced long words in type $A$ and $C$ that there is at most one $j'<k<j-1$ such that $\la \al_{i_k},\al_{i_{j}}^\vee \ra \neq 0$. First, if the sum in the last line is $0$, then $b_{j'}$ is both circled and boxed, and the pattern is indeed unstable. Otherwise, the only surviving term is $-b_{k'}\la \al_{i_{k'}},\al_{i_j}^\vee\ra=b_{k'}$ for $b_{k'}$ in the same row as $b_{j'}$. (This can be verified directly for type $A$ and $C$, but does not naively generalize to other types.) In this case, the pattern is again unstable.
\end{proof}

Thus, in each row of a stable pattern, the decoration is of the following form
\[
\square \ \square \  \cdots \ \square \ \bigcirc \ \cdots \ \ \bigcirc
\]
where the circled entries are zero.

Let $v\in \scb_\lam$ such that $\BZL(v)$ is a stable pattern. We now define a Weyl group element using Proposition \ref{prop:bijection type A C} and the decoration of $\BZL(v)$. Define $\mathrm{sign}(v)=(e_1,\cdots, e_N)\in \{0,1\}^N$ where $e_i=1$ if $b_i$ is boxed.
Let $St_\lam$ be the set of stable patterns. Define a map $St_\lam\to W$ by
\[
\BZL(v)\mapsto \mathrm{sign}(v)=(e_1,\cdots,e_N)\mapsto w_v:= s_{i_1}^{e_1} s_{i_2}^{e_2} \cdots s_{i_N}^{e_N}.
\]
We also write
\begin{equation}\label{eq:reduced of w}
w_v=s_{i'_1}\cdots s_{i'_\ell}
\end{equation}
to keep track of the simple reflections that actually appear in $w_v$.  By Proposition \ref{prop:bijection type A C}, this is a reduced expression of $w_v$ and the map $St_\lam \to W$ is injective. We now show that this map is indeed surjective. We also show that the weight of $v$ is $w_v\lam$; in particular, this shows that $\mathrm{wt}(v)$ lies in the orbit of the highest weight $\lam$.



\begin{Prop}\label{prop:bijection} \
\begin{enumerate}
\item The map $St_\lam\to W$ is surjective.
\item Let $\BZL(v)$ be a stable pattern. Then $\mathrm{wt}(v)=w_v\lam$.
\end{enumerate}
\end{Prop}

\begin{proof}
Given an element $w\in W$, we explain how to reconstruct a stable pattern $\BZL(v)$ such that $w_v=w$ and $\mathrm{wt}(v)=w_v\lam$. As the decoration for the pattern comes from the decomposition of $w$, we only need to specify the nonzero entries. We write $w=s_{i'_1}\cdots s_{i'_\ell}$ as in \eqref{eq:reduced of w}. We reindex the nonzero entries as
 \[
 (b'_1,\cdots,b'_\ell)
 \]
 where $\ell$ is the length of $w$.

Each nonzero entry must be the upper bound of its polytope inequality. Therefore, $(b_1',\cdots, b'_\ell)$ can be defined inductively via
\[
\begin{aligned}
b'_\ell&=\la \lam,\al_{i'_\ell}^\vee\ra\\
b'_{\ell-1}&= \la \lam-b'_{\ell }\al_{i'_\ell},\al_{i'_{\ell-1}}^\vee \ra\\
\cdots & \\
b'_1&=\la \lam-\sum_{k=2}^\ell b'_k\al_{i'_k},\al_{i'_1}^\vee\ra.
\end{aligned}
\]
Now every entry has at least one decoration. Adding zero entries appropriately gives the desired pattern $\BZL(v_w)$. Here $v_w$ is the corresponding vertex in $\scb_\lam$. To verify this is a stable pattern, we need to show that every entry has exactly one decoration.

Before proving that, we calculate $\mathrm{wt}(v_w)$. By properties of crystal graph, one can find $v_w$ from $v_\lam$ by the following path
\[
v_{\lam}, \ f_{{i_N}}^{b_N}(v_{\lam}), \ f_{{i_{N-1}}}^{b_{N-1}}f_{{{i_N}}}^{b_N}(v_{\lam}), \ \cdots
\]
If we would like to ignore the zeros in $\BZL(v_w)$, we may write $w=s_{i'_1}\cdots s_{i'_\ell}$ and one can find $v_w$ by the following path
\[
v_{\lam}, \ f_{{i'_\ell}}^{b'_\ell}(v_{\lam}), \ f_{{i'_{\ell-1}}}^{b'_{\ell-1}}f_{{i'_\ell}}^{b'_\ell}(v_{\lam}), \ \cdots
\]
We claim that
\begin{equation}\label{eq:bj value}
b'_j=\la s_{i'_{j+1}}\cdots s_{i'_\ell}(\lam),\al_{i'_j}^\vee\ra, \qquad \mathrm{wt}(f_{i'_{j}}^{b'_{j}}\cdots f_{{s_{i'_\ell}}}^{b'_\ell}(v_{\lam}))=s_{i'_{j}}\cdots s_{i'_\ell}(\lam).
\end{equation}

We argue by reduction on $j$. If $j=\ell$, then $b'_\ell=\la \lam, \al_{i'_\ell}^\vee \ra$ and
\[
\mathrm{wt}(f_{{i'_\ell}}^{b'_\ell}(v_{\lam}))=\lam - b'_\ell \al_{i'_\ell}=s_{i'_\ell}(\lam).
\]
Now assume that the result is true for $j+1,\cdots, \ell$ and we prove it for $j$. 
For a root $\al$,
\begin{equation}\label{eq:calculation long}
\begin{aligned}
&\la s_{i'_{j+1}}\cdots s_{i'_\ell}(\lam),\al^\vee\ra \\
=&\la s_{i'_{j+2}}\cdots s_{i'_\ell}(\lam),  s_{i'_{j+1}}(\al^\vee)\ra\\
=&\la s_{i'_{j+2}}\cdots s_{i'_\ell}(\lam),  \al^\vee-\la \al^\vee ,\al_{i'_{j+1}}\ra \al_{i'_{j+1}}^\vee\ra\\
=&\la s_{i'_{j+2}}\cdots s_{i'_\ell}(\lam),  \al^\vee\ra - b'_{j+1}\la \al^\vee ,\al'_{i_{j+1}}\ra \\
=&\la s_{i'_{j+3}}\cdots s_{i'_\ell}(\lam),  \al^\vee\ra - b'_{j+1}\la \al^\vee ,\al_{i'_{j+1}}\ra - b'_{j+2}\la \al^\vee ,\al_{i'_{j+2}}\ra \\
&\cdots \\
=&\left\la \lam - \sum_{k=j+1}^\ell b'_k\al_{i'_k},\al^\vee \right\ra.
\end{aligned}
\end{equation}
Therefore, $\la s_{i'_{j+1}}\cdots s_{i'_\ell}(\lam),\al_{i'_j}^\vee\ra=b'_{j}$.
The weight of $f_{i'_{j}}^{b'_j}\cdots f_{i'_\ell}^{b'_\ell}(v_{\lam})$ is
\[
s_{i'_{j+1}}\cdots s_{i'_\ell}(\lam)- b'_j\al_{i'_j}=s_{i'_{j+1}}\cdots s_{i'_\ell}(\lam)- \la s_{i'_{j+1}}\cdots s_{i'_\ell}(\lam),\al_{i'_j}^\vee\ra \al_{i'_j}=s_{i'_{j}}\cdots s_{i'_\ell}(\lam).
\]
This finishes the induction step.

We now verify that $\BZL(v_w)$ is a stable pattern. Suppose there is an entry $b_j$ such that it is boxed and circled. If it is a zero entry, then
\[
0=b_j=\left\la\lam -\sum_{k=j+1}^N b_k\al_{i_k},\al_{i_j}^\vee\right\ra.
\]
However, by \eqref{eq:calculation long}, the right-hand side is $\la \tilde w(\lam),\al_{i_j}^\vee\ra=\la \lam , \tilde w^{-1}(\al_{i_j}^\vee)\ra$ for some Weyl group element $\tilde w$. This cannot be zero as we assume $\la \lam,\al_i^\vee\ra>0$ for all simple coroots (recall that any coroot is either a non-negative or non-positive linear combination of simple coroots). The case $b_j\neq 0$ can be analyzed analogously. This proves part (1). Part (2) now follows from \eqref{eq:bj value} and the fact that $St_\lam \to W$ is a bijection.
\end{proof}

\section{Weyl Group Multiple Dirichlet Series}\label{sec:WGMDS}

\subsection{Definition}

Given an isotropic subgroup $\Omega$ of $F_S^\times$, let $\scm(\Omega^r)$ be the space of functions $\Psi:(F_S^\times)^r\to \bc$ that satisfy the transformation property
\[
\Psi(\boldsymbol{\epsilon}\boldsymbol{c})=
\left(
\prod_{i=1}^r (\epsilon_i,c_i)_S^{\lVert \al_i \rVert^2} \prod_{i< j}(\epsilon_i,c_j)_S^{2\la \al_i,\al_j\ra}
\right)
\Psi(\boldsymbol{c})
\]
for all $\boldsymbol{\epsilon}=(\epsilon_1,\cdots,\epsilon_r)\in\Omega^r$ and all $\boldsymbol{c}=(c_1,\cdots,c_r)\in (F_S^\times )^r$.

Given a root system $\Phi$ of fixed rank $r$, and integer $n\geq 1$, $\boldsymbol{m}\in\sco_S^r$, and $\Psi\in\scm(\Omega^r)$, we consider a function of $r$ complex variables $\boldsymbol{s}=(s_1,\cdots,s_r)\in\bc^r$ of the form
\[
Z_\Psi(s_1,\cdots, s_r;m_1,\cdots, m_r)=Z_\Psi(\boldsymbol{s};\boldsymbol{m})=\sum_{\boldsymbol{c}
=(c_1,\cdots,c_r)\in(\sco_S/\sco_S^\times)^r} \dfrac{H^{(n)}(\boldsymbol{c};\boldsymbol{m}) \Psi(\boldsymbol{c})}{|c_1|^{2s_1}\cdots |c_r|^{2s_r}}.
\]

The function $H(\boldsymbol{c};\boldsymbol{m})=H^{(n)}(\boldsymbol{c};\boldsymbol{m})$ carries the main arithmetic content. It is not defined as a multiplicative function, but rather a ``twisted multiplicative'' function. This means that for $S$-integer vectors $\boldsymbol{c},\boldsymbol{c}'\in(\sco_S/\sco_S^\times)^r$ with $\gcd(c_1\cdots c_r,c_1'\cdots c_r')=1$,
\[
H(c_1c_1',\cdots, c_rc_r';\boldsymbol{m})=\mu(\boldsymbol{c},\boldsymbol{c}') H(\boldsymbol{c};\boldsymbol{m})H(\boldsymbol{c}';\boldsymbol{m}),
\]
where $\mu(\boldsymbol{c},\boldsymbol{c}')$ is an $n$-th root of unity given by
\[
\mu(\boldsymbol{c},\boldsymbol{c}')=
\prod_{i=1}^r
\left(\dfrac{c_i}{c_i'}\right)_n^{\lVert \al_i\rVert^2}
\left(\dfrac{c_i'}{c_i}\right)_n^{\lVert \al_i\rVert^2}
\prod_{i<j}
\left(\dfrac{c_i}{c_j'}\right)_n^{2\la \al_i,\al_j\ra}
\left(\dfrac{c_i'}{c_j}\right)_n^{2\la \al_i,\al_j\ra}.
\]

The transformation property of functions in $\scm(\Omega^r)$ implies that
\[
H(\boldsymbol{\epsilon c};\boldsymbol{m})\Psi(\boldsymbol{\epsilon c})
=H(\boldsymbol{c};\boldsymbol{m})\Psi(\boldsymbol{ c})
\qquad
\text{ for all }\boldsymbol{\epsilon}\in\sco_S^r,\boldsymbol{c},\boldsymbol{m}\in (F_S^\times)^r.
\]
The proof requires the $n$-th power reciprocity law (\cite{BBF-untwisted-stable} Lemma 1.2).

Now, given any $\boldsymbol{m},\boldsymbol{m}',\boldsymbol{c}\in\sco_S^r$ with $\gcd(m_1'\cdots m_r',c_1\cdots c_r)=1$, let
\[
H(\boldsymbol{c};m_1m_1',\cdots, m_rm_r')
=\prod_{i=1}^r
\left( \dfrac{m_i'}{c_i} \right)_n^{-\lVert \al_i\rVert^2}
H(\boldsymbol{c};\boldsymbol{m}).
\]

Thus, it is enough to specify the coefficients $H(p^{\boldsymbol{k}};p^{\boldsymbol{l}}):=H(p^{k_1},\cdots,p^{k_r};p^{l_1},\cdots,p^{l_r})$ for any fixed prime $p$ with $\boldsymbol{k}=(k_1,\cdots,k_r)$, $\boldsymbol{l}=(l_1,\cdots,l_r)$, $k_i=\mathrm{ord}_p(c_i)$ and $l_i=\mathrm{ord}_p(m_i)$ in order to completely determine $H(\boldsymbol{c},\boldsymbol{m})$ for any pair of $S$-integer vectors $\boldsymbol{m}$ and $\boldsymbol{c}$.

The goal in the theory of Weyl group multiple Dirichlet series is to find $H(p^{\boldsymbol{k}};p^{\boldsymbol{l}})$ so that $Z_\Psi(\boldsymbol{s};\boldsymbol{m})$ admits meromorphic continuation to $\bc^r$ and satisfies a group of functional equations which is isomorphic to $W$.

At mentioned in the introduction, there are at least five ways to define the prime-power contribution and they are expected to agree.
In this section, we compare the Lie-theoretic description and the crystal description for type $A$ and type $C$ ($n$ odd). The comparison is given in \cite{BBF-twisted} Section 8 for type A, and \cite{BBF-typec} Section 4 for type $C$ and $n$ odd. In both cases, the proofs rely on explicit realization of the root systems in $\br^r$.

In \cite{BBF-typec} Section 4.1, the authors ask if one can compare the Lie-theoretic description and the crystal graph description in the stable range without the need to refer to an explicit embedding of the underlying root system in an Euclidean space. Using Proposition \ref{prop:bijection}, we will present such a coordinate-free proof in this section.

\subsection{Twisted Weyl group multiple Dirichlet series: the stable case}\label{sec:stable case}
In this section, we recall the definition of twisted Weyl group multiple Dirichlet series in the stable range in \cite{BBF-twisted}.

Fix non-negative integers $l_1,\cdots, l_r$ and let $\lam=\sum l_i \omega_i$ be the corresponding weight.

\begin{SA}
\textit{The positive integer $n$ satisfies the following property. Let $\al_0$ be the highest root and let $\al_0^\vee=\sum_{i=1}^r t_i \al_i^\vee$ be the highest coroot in the partial ordering. Then
\[
n\geq \gcd(n,\lVert\al_0\rVert^2) \cdot d_\lam(\al_0)=\gcd(n,\lVert\al_0\rVert^2) \cdot \sum_{i=1}^r t_i(l_i+1),
\]
where $d_\lam(\al_0)=\la \lam+\rho, \al_0^\vee \ra$ as defined in \eqref{eq:d lambda}.
}
\end{SA}

\begin{Rem}
Notice that there is misprint in the Stability Assumption in \cite{BBF-twisted} Section 3, \cite{BBF-typec} Section 4, and \cite{FZ-orthogonal} Section 9. One should use $t_i$ from the highest coroot instead of the ones from the highest root.
\end{Rem}

\begin{Lem}[\cite{BBF-twisted}, Lemma 1]
Let $w\in W$.
\begin{enumerate}
\item There are non-negative integers $k_i$ such that
\begin{equation}\label{eq:definition of ki}
\rho+\lam-w(\rho+\lam)=\sum_{i=1}^r k_i \al_i.
\end{equation}
\item If $w,w'\in W$ such that $\rho+\lam-w(\rho+\lam)=\rho+\lam-w'(\rho+\lam)$ then $w=w'$.
\end{enumerate}
\end{Lem}

Recall that by twisted multiplicativity, it remains to describe $H(p^{\boldsymbol{k}};p^{\boldsymbol{l}})$ for any fixed prime $p$. For any given $(k_1,\cdots, k_r)$, these coefficients are defined to be zero unless there exists $w\in W$  such that \eqref{eq:definition of ki} holds. In this case,  we  define
\[
H(p^{\boldsymbol{k}};p^{\boldsymbol{l}})=\prod_{\al\in\Phi(w)}g_{\lVert \al\rVert^2}(p^{d_{\lam}(\al)-1},p^{d_{\lam}(\al)}).
\]

\subsection{Crystal graph description}
The prime-power construction $H(p^{\boldsymbol{k}};p^{\boldsymbol{l}})$ can also be described as a sum of $n$-th order Gauss sum over the crystal graph $\scb_{\lam+\rho}$:
\[
H(p^{\boldsymbol{k}};p^{\boldsymbol{l}})=\sum_{\substack{v\in\scb_{\lam+\rho} \\ \mathrm{wt}(v)=\mu}} G(v).
\]
Here the relation between $(k_1,\cdots, k_r)$ and $\mu$ is
\[
\rho+\lam-w(\mu)=\sum_{i=1}^r k_i \al_i.
\]
The definition here depends on the choice of a reduced decomposition of $w_0$. Here we choose the nice decompositions in Proposition \ref{prop:bijection type A C}.

A definition of $G(v)$ is already given for type $A$ (\cite{BBF-crystal}) and type  $C$ with odd $n$ (\cite{FZ-orthogonal} (34)). We remark that in these two papers the primer-power contribution is defined via the Gelfand-Tsetlin patterns. This description is equivalent to the crystal graph description; see \cite{BBF-crystal} and \cite{BBF-typec-crystal}.  As explained in Remark \ref{rem:different description}, we take slightly different descriptions to make the comparison more natural.

The definition of $G(v)$ is given as follows:
\begin{equation}\label{eq:crystal contribution}
G(v)=\prod_{b_j\in \BZL(v)}
\begin{cases}
q^{b_j} & \text{ if } b_j \text{ is circled but not boxed}, \\
g_{\lVert \al_{i_j} \rVert^2}(p^{b_j-1},p^{b_j}) & \text{ if } b_j \text{ is boxed but not circled}, \\
q^{b_j}(1-q^{-1})& \text{ if } b_j \text{ is neither circled nor boxed and }n\mid b_j,\\
0 & \text{ otherwise}.\\
\end{cases}
\end{equation}
Here, the index $i_j$ is the one appearing in the nice decomposition of $w_0$.

It is still not clear how to define $G(v)$ in general. For Eisenstein series constructed from an automorphic representation on a braidless maximal parabolic subgroup, an analogous calculation has been carried out in \cite{BF-GAFA}.

\subsection{Comparison}\label{sec:comparison}

In this section, we only work with roots system of type $A$ or type $C$ with odd $n$. We show that in these two cases, the crystal graph description in the stable range agrees with the Lie-theoretic description.  We write $H_{\BZL}(p^{\boldsymbol{k}};p^{\boldsymbol{l}})$ for the crystal graph description, and $H_{St}(p^{\boldsymbol{k}};p^{\boldsymbol{l}})$ for the Lie-theoretic description.

\begin{Thm}
Let $\Phi$ be a root system of type $A$ or $C$. Let $n$ be a positive integer satisfying the stability assumption. We also require $n$ to be odd if $\Phi$ is of type $C$.
\begin{enumerate}
\item If $v\in\scb_{\rho+\lam}$ such that $\BZL(v)$ is an unstable pattern, then $G(v)=0$.
\item Let $w\in W$. Let $\boldsymbol{k}$ be defined in \eqref{eq:definition of ki}. Then
\[
H_{\BZL}(p^{\boldsymbol{k}};p^{\boldsymbol{l}})=H_{St}(p^{\boldsymbol{k}};p^{\boldsymbol{l}}).
\]
\end{enumerate}
\end{Thm}

\begin{proof}
Suppose $\BZL(v)$ is an unstable pattern. If there is an entry which is both circled and boxed, then $G(v)=0$ from the fourth case in \eqref{eq:crystal contribution}. If one of the entries $b_j$ is neither circled nor boxed, then by Lemma  \ref{lem:unstable},  $b_j<\la \lam,\al_0^\vee\ra\leq n$. In other words, $n\nmid b_j$. Again, by the fourth case in \eqref{eq:crystal contribution}, we conclude that $G(v)=0$.

We are now left with the stable patterns. By Proposition \ref{prop:bijection}, they are in bijection with $W$.  Given $w\in W$, let $\BZL(v_w)$ be the corresponding stable pattern and $v_w$ be the corresponding vertex in $\scb_{\rho+\lam}$. Recall that the weight of $v_w$ is $w(\rho+\lam)$. The pattern $\BZL(v_w)$ contributes to $H_{\BZL}(p^{\boldsymbol{k}};p^{\boldsymbol{l}})$ where $\rho+\lam-w(\rho+\lam)=\sum k_i\al_i$. This is also the only pattern contributing to this term.  On the other hand, $w$ contributes to $H_{St}(p^{\boldsymbol{k}};p^{\boldsymbol{l}})$ with the same $\boldsymbol{k}$.

To calculate $H_{St}(p^{\boldsymbol{k}};p^{\boldsymbol{l}})$, we write $w=s_{i'_1}\cdots s_{i'_\ell}$ as in \eqref{eq:reduced of w}. Define
\[
\beta_\ell=\al_{i'_\ell}, \ \beta_{\ell-1}=s_{i'_\ell}(\al_{i'_{\ell-1}}), \ \cdots, \ \beta_1=s_{i'_\ell}s_{i'_{\ell-1}}\cdots s_{i'_2}(\al_{i'_1}).
\]
Then by Lemma \ref{lem:bump lemma},
\[
\Phi(w)=\{\beta_\ell, \ \cdots, \ \beta_1\}.
\]
Note that $\lVert \beta_j\rVert^2=2$ if and only if $\lVert \al_{i'_j}\rVert^2=2$. Thus the contribution of $w$ is
\[
H_{St}(p^{\boldsymbol{k}};p^{\boldsymbol{l}})=\prod_{j=1}^\ell g_{\lVert \beta_j\rVert^2} (p^{d_\lam(\beta_j)-1}, p^{d_\lam(\beta_j)}).
\]

We now calculate $H_{\BZL}(p^{\boldsymbol{k}};p^{\boldsymbol{l}})$. Write $\BZL(v_w)=(b_1,\cdots,b_N)$. Let $(b_1',\cdots,b'_\ell)$ be the nonzero entries in $\BZL(v_w)$. By \eqref{eq:bj value},
\[
(b'_1,\cdots, b'_\ell)=(d_{\lam}(\beta_1),\cdots, d_{\lam}(\beta_\ell)).
\]
 Note that the circled entries are zero. Thus
\[
H_{\BZL}(p^{\boldsymbol{k}};p^{\boldsymbol{l}})=G(v_w)=\prod_{b_j\in \BZL(v_w),b_j\neq 0}g_{\lVert \al_{i_j} \rVert^2}(p^{b_j-1},p^{b_j})=\prod_{j=1}^\ell g_{\lVert \al_{i'_j}\rVert^2}(p^{b'_j-1},p^{b'_j}),
\]
where the index $i_j$ is the one appearing in the nice decomposition of $w_0$. The proof is now complete.
\end{proof}

\begin{Rem}\label{rem:general description}
We now say a few words regarding the crystal graph description in general, at least in the stable range.

The proof here uses the following ingredients:
\begin{itemize}
  \item Some estimates of the unstable patterns, which help us eliminate such patterns.
  \item The natural bijection between $W$ and the set of stable patterns (both of them depends on a choice of nice decomposition of the longest element). The contribution is calculated using the first two formula in \eqref{eq:crystal contribution}.
\end{itemize}

In the stable range of the `correct' crystal graph description for other root systems admitting a good enumeration, one should take the corresponding formulas in \eqref{eq:crystal contribution}. The proof presented here works without essential change.
\end{Rem}

\appendix

\section{Examples}\label{app:examples}

In this appendix,  we work out $T_{\tau_\ell}$ (see Remark \ref{rem:graph}) and the bijection in Corollary \ref{cor:main}  for all the braidless weights.

\subsection{Type A}\label{sec:example type A}

Let $\fg=\fsl_{r+1}$. Any enumeration is a good enumeration for $\fg$. So we simply choose the following one:
\[
\begin{tikzpicture}
        \node (a1) at (-2,0) {$\al_1$};
        \node (a2) at (-1,0) {$\al_2$};
        \node (a3) at (0,0) {$\al_3$};
        \node (a4) at (1,0) {$\cdots$};
        \node (a5) at (2,0) {$\al_r$};
  \draw [-] (a1) to (a2);
  \draw [-] (a2) to (a3);
  \draw [-] (a3) to (a4);
  \draw [-] (a4) to (a5);
\end{tikzpicture}
\]

Let $\omega=\omega_1$.  Then
$\tau_\ell=s_1 s_{2}\cdots s_r\in {}^\omega W$. The graph $T_{\tau_\ell}$ is
\[
s_1\rightarrow s_{2} \rightarrow \cdots \rightarrow s_r.
\]
and
${}^\omega W=\{I,  \ s_1, \ s_1 s_{2}, \ \cdots, \ s_1 s_{2} \cdots s_r \}.$

If we pick a general $\omega=\omega_j$, then
\[
\tau_\ell=(s_js_{j-1}\cdots s_1)(s_{j+1} s_j\cdots s_2)\cdots (s_r s_{r-1}\cdots s_{r-j+1}).
\]
The graph $T_{\tau_\ell}$ is
\[
        \begin{tikzpicture}
        \node (00) at (0,0) {$s_j$};
        \node (01) at (1.5,0) {$s_{j+1}$};
        \node (02) at (3,0) {$\cdots$};
        \node (03) at (4.5,0) {$s_r$};

      \node (10) at (0,1) {$s_{j-1}$};
        \node (11) at (1.5,1) {$s_{j}$};
        \node (12) at (3,1) {$\cdots$};
        \node (13) at (4.5,1) {$s_{r-1}$};

        \node (20) at (0,2) {$\vdots$};
        \node (21) at (1.5,2) {$\vdots$};
                \node (22) at (3,2) {$\iddots$};

        \node (23) at (4.5,2) {$\vdots$};

        \node (30) at (0,3) {$s_1$};
        \node (31) at (1.5,3) {$s_{2}$};
        \node (32) at (3,3) {$\cdots$};
        \node (33) at (4.5,3) {$s_{r-j+1}$};

        \draw [->] (00) to (01);
        \draw [->] (01) to (02);
        \draw [->] (02) to (03);

        \draw [->] (10) to (11);
        \draw [->] (11) to (12);
        \draw [->] (12) to (13);

        \draw [->] (30) to (31);
        \draw [->] (31) to (32);
        \draw [->] (32) to (33);

        \draw [->] (00) to (10);
        \draw [->] (10) to (20);
        \draw [->] (20) to (30);

        \draw [->] (01) to (11);
        \draw [->] (11) to (21);
        \draw [->] (21) to (31);

        \draw [->] (03) to (13);
        \draw [->] (13) to (23);
        \draw [->] (23) to (33);
                        \end{tikzpicture}.
\]
Using Corollary \ref{cor:main}, it is easy to show that there is a bijection between
\[
\{
(i_j,\cdots, i_k,0,\cdots,0)\in\bz^{r-j+1}: j-1\leq k\leq r, \ 1\leq i_j<i_{j+1}<\cdots <i_k\leq k
\}
\]
and ${}^\omega W$. (When $k=j-1$, we only consider the element $(0,0,\cdots,0)\in\bz^{r-j+1}$.)
The map is given by
\[
(i_j,\cdots, i_k,0,\cdots,0)\mapsto (s_js_{j-1}\cdots s_{i_j})(s_{j+1} s_j\cdots s_{i_{j+1}})\cdots (s_k s_{k-1}\cdots s_{i_k}).
\]
(See also \cite{Cai18} Theorem 2.4 for a description in terms of Young tableaux.)

\subsection{Type B, C}\label{sec:example type BC}

For $\Phi=B_r$ or $C_r$, the good enumeration in \cite{Littelmann98} is given as follows:
\[
 \begin{tikzpicture}
        \node (a1) at (-2,0) {$\al_1$};
        \node (a2) at (-1,0) {$\al_2$};
        \node (a3) at (0,0) {$\al_3$};
        \node (a4) at (1,0) {$\cdots$};
        \node (a5) at (2,0) {$\al_r$};
\draw [-] (-1.65,0.04) to (-1.35,0.04);
\draw [-] (-1.65,-0.04) to (-1.35,-0.04);
  \draw [-] (a2) to (a3);
  \draw [-] (a3) to (a4);
  \draw [-] (a4) to (a5);
                        \end{tikzpicture}.
\]

If  $\omega=\omega_r$, then
\[
\tau_\ell=s_rs_{r-1}\cdots s_2s_1s_2 \cdots s_{r-1} s_r.
\]
The graph $T_{\tau_\ell}$ is
\[
s_r\to s_{r-1}\to \cdots \to s_2\to s_1\to s_2\to \cdots \to s_{r-1} \to s_r,
\]
and
${}^\omega W=\{I, \ s_r,\  s_rs_{r-1}, \ \cdots, \ s_rs_{r-1}\cdots s_1,\ s_rs_{r-1}\cdots s_1s_2, \ \cdots, \ \tau_\ell\}.$

If $\omega=\omega_1$, then
\[
\tau_\ell=(s_1s_2\cdots s_r)(s_1s_2\cdots s_{r-1})\cdots (s_1).
\]
The graph $T_{\tau_\ell}$ is
\[
        \begin{tikzpicture}
        \node (11) at (1,1) {$s_1$};
        \node (21) at (1,2) {$s_2$};
        \node (31) at (1,3) {$s_3$};
        \node (41) at (1,4) {$\vdots$};
        \node (51) at (1,5) {$s_r$};

        \node (22) at (2,2) {$s_1$};
        \node (32) at (2,3) {$s_2$};
     \node (42) at (2,4) {$\vdots$};
     \node (52) at (2,5) {$s_{r-1}$};

     \node (33) at (3,3) {$s_1$};
     \node (43) at (3,4) {$\iddots$};
     \node (53) at (3,5) {$\cdots$};

     \node (44) at (4,4) {$\iddots$};
     \node (54) at (4,5) {$s_2$};

     \node (55) at (5,5) {$s_1$};

        \draw [->] (11) to (21);
  \draw [->] (21) to (31);
  \draw [->] (31) to (41);
  \draw [->] (41) to (51);
  \draw [->] (44) to (54);

  \draw [->] (22) to (32);
  \draw [->] (32) to (42);
  \draw [->] (42) to (52);

  \draw [->] (21) to (22);
  \draw [->] (31) to (32);
  \draw [->] (32) to (33);
  \draw [->] (33) to (43);
  \draw [->] (51) to (52);
  \draw [->] (52) to (53);
\draw [->] (53) to (54);
\draw [->] (54) to (55);

\end{tikzpicture}
\]
There is a bijection
\[
\{
(i_1,\cdots, i_k,0,\cdots, 0)\in\bz^r: 0\leq k\leq r, \ r\geq i_1>\cdots >i_k>0
\}
\to  {}^\omega W,
\]
where the map is given by
\[
(i_1,\cdots, i_k,0,\cdots, 0)\mapsto (s_1s_2 \cdots s_{i_1})(s_1 s_2 \cdots s_{i_2}) \cdots (s_1 s_2\cdots s_{i_k}).
\]

\begin{Rem}
This parametrization is slightly different from the parametrization in \cite{GPSR} page 20. With our parametrization, one can show that
\[
\Phi(w)=\bigsqcup_{l=1}^k \Phi(w,l), 
\]
where $\Phi(w,l)$ is the set of roots corresponding to the first $i_l$ roots (counting from the bottom) in the $(r+l)$th column of $Sp_{2r}$. This parametrization may simplify the calculation in \cite{GPSR} Section 5.
\end{Rem}
\subsection{Type D}\label{sec:type d}

The enumeration in \cite{Littelmann98} is given as follows:
\[
 \begin{tikzpicture}
        \node (a1) at (-.707,0.707) {$\al_1$};
        \node (a2) at (-.707,-0.707) {$\al_2$};
        \node (a3) at (0,0) {$\al_3$};
        \node (a4) at (1,0) {$\cdots$};
        \node (a5) at (2,0) {$\al_r$};
  \draw [-] (a1) to (a3);
  \draw [-] (a2) to (a3);
  \draw [-] (a3) to (a4);
  \draw [-] (a4) to (a5);
                        \end{tikzpicture}
\]

If $\omega=\omega_r$, then
$\tau_\ell=s_rs_{r-1}\cdots s_3s_1s_2s_3 \cdots s_{r-1} s_r.$
The graph $T_{\tau_\ell}$ is
\[
        \begin{tikzpicture}
        \node (A-1) at (-4.2,0) {$s_{r}$};
        \node (A-2) at (-3,0) {$s_{r-1}$};
        \node (A) at (-1.707,0) {$\cdots$};
        \node (B) at (-0.707,0) {$s_3$};
        \node (C) at (0,0.707) {$s_1$};
        \node (D) at (0,-0.707) {$s_2$};
        \node (E) at (0.707,0) {$s_3$};
        \node (F) at (1.707,0) {$\cdots$};
        \node (F+1) at (3,0) {$s_{r-1}$};
        \node (F+2) at (4.2,0) {$s_{r}$};
                        \draw [->] (A) to (B);
  \draw [->] (B) to (C);
  \draw [->] (B) to (D);
  \draw [->] (C) to (E);
  \draw [->] (D) to (E);
  \draw [->] (E) to (F);
  \draw [->] (A-1) to (A-2);
  \draw [->] (A-2) to (A);
  \draw [->] (F) to (F+1);
  \draw [->] (F+1) to (F+2);
                        \end{tikzpicture},
\]
and
\[
\begin{aligned}
{}^\omega W=\{ & I, \ s_r, s_rs_{r-1}, \ \cdots, \ s_rs_{r-1}\cdots s_3s_1, \ s_rs_{r-1}\cdots s_3 s_2, \\
&s_rs_{r-1}\cdots s_3s_1s_2, \ s_rs_{r-1}\cdots s_3s_1s_2s_3, \ \cdots, \ \tau_\ell\}.
\end{aligned}
\]

If $\omega=\omega_1$ (the case $\omega=\omega_2$ is similar), then
\[
\tau_\ell=(s_1s_3s_4\cdots s_{r})(s_2s_3\cdots s_{r-1}) (s_1s_3\cdots s_{r-2})(s_1s_3\cdots s_{r-3})\cdots (s_1).
\]
The graph $T_{\tau_\ell}$ is
\[
        \begin{tikzpicture}
        \node (11) at (1,1) {$s_1$};
        \node (21) at (1,2) {$s_3$};
        \node (31) at (1,3) {$s_4$};
        \node (41) at (1,4) {$\vdots$};
        \node (51) at (1,5) {$s_r$};

        \node (22) at (2,2) {$s_2$};
        \node (32) at (2,3) {$s_3$};
     \node (42) at (2,4) {$\vdots$};
     \node (52) at (2,5) {$s_{r-1}$};

     \node (33) at (3,3) {$s_1$};
     \node (43) at (3,4) {$\iddots$};
     \node (53) at (3,5) {$\cdots$};

     \node (44) at (4,4) {$s_1$};
     \node (54) at (4,5) {$s_3$};

     \node (55) at (5,5) {$s_1$};

        \draw [->] (11) to (21);
  \draw [->] (21) to (31);
  \draw [->] (31) to (41);
  \draw [->] (41) to (51);
  \draw [->] (44) to (54);

  \draw [->] (22) to (32);
  \draw [->] (32) to (42);
  \draw [->] (42) to (52);

  \draw [->] (21) to (22);
  \draw [->] (31) to (32);
  \draw [->] (32) to (33);
   \draw [->] (33) to (43);
  \draw [->] (43) to (44);
  \draw [->] (51) to (52);
  \draw [->] (52) to (53);
\draw [->] (53) to (54);
\draw [->] (54) to (55);

                        \end{tikzpicture}.
\]
There is a bijection between
\[
\{
(i_1,\cdots, i_k,0,\cdots, 0)\in \bz^{r-1}: 0\leq k\leq r-1,r\geq i_1>\cdots >i_k>0, \ i_k\neq 2 \text{ if }k\neq 2, \ i_2\neq 1
\}\]
and  ${}^\omega W$,
where the map is given by
\[
(i_1,\cdots, i_k,0,\cdots, 0)\mapsto (s_1s_3 \cdots s_{i_1})(s_2 s_3 \cdots s_{i_2})(s_1s_3 \cdots s_{i_3}) \cdots (s_1 s_3\cdots s_{i_k}).
\]

The remark in the previous section applies in this case as well.

\subsection{Type E}

For $E_6$ and $E_7$, the good enumerations in \cite{Littelmann98} are given as follows:
\[
        \begin{tikzpicture}
        \node (a5) at (-2,0) {$\al_5$};
        \node (a4) at (-1,0) {$\al_4$};
        \node (a3) at (0,0) {$\al_3$};
        \node (a2) at (1,0) {$\al_2$};
        \node (a6) at (2,0) {$\al_6$};
        \node (a1) at (0,1) {$\al_1$};

                        \draw [-] (a5) to (a4);
  \draw [-] (a4) to (a3);
  \draw [-] (a3) to (a2);
  \draw [-] (a3) to (a1);
  \draw [-] (a2) to (a6);
                        \end{tikzpicture},
                        \qquad
                         \begin{tikzpicture}
        \node (a5) at (-2,0) {$\al_5$};
        \node (a4) at (-1,0) {$\al_4$};
        \node (a3) at (0,0) {$\al_3$};
        \node (a2) at (1,0) {$\al_2$};
        \node (a6) at (2,0) {$\al_6$};
        \node (a7) at (3,0) {$\al_7$};
        \node (a1) at (0,1) {$\al_1$};

                        \draw [-] (a5) to (a4);
                          \draw [-] (a4) to (a3);
                          \draw [-] (a3) to (a2);
                          \draw [-] (a3) to (a1);
                          \draw [-] (a2) to (a6);
                          \draw [-] (a7) to (a6);
                        \end{tikzpicture}.
\]

When $\Phi=E_6$, we choose $\omega=\omega_6$. When $\Phi=E_7$, we choose $\omega=\omega_7$.
The longest elements $\tau_\ell$ are
\[
\begin{aligned}
E_6: & \ s_6s_2s_3s_1s_4s_5s_3s_4s_2s_3s_1s_6s_2s_3s_4s_5\\
E_7: & \ s_7s_6s_2s_3s_1s_4s_5s_3s_4s_2s_3s_1s_6s_2s_3s_4s_5s_7s_6s_2s_3s_1s_4s_3s_2s_6s_7.
 \end{aligned}
\]
Here are the graphs:
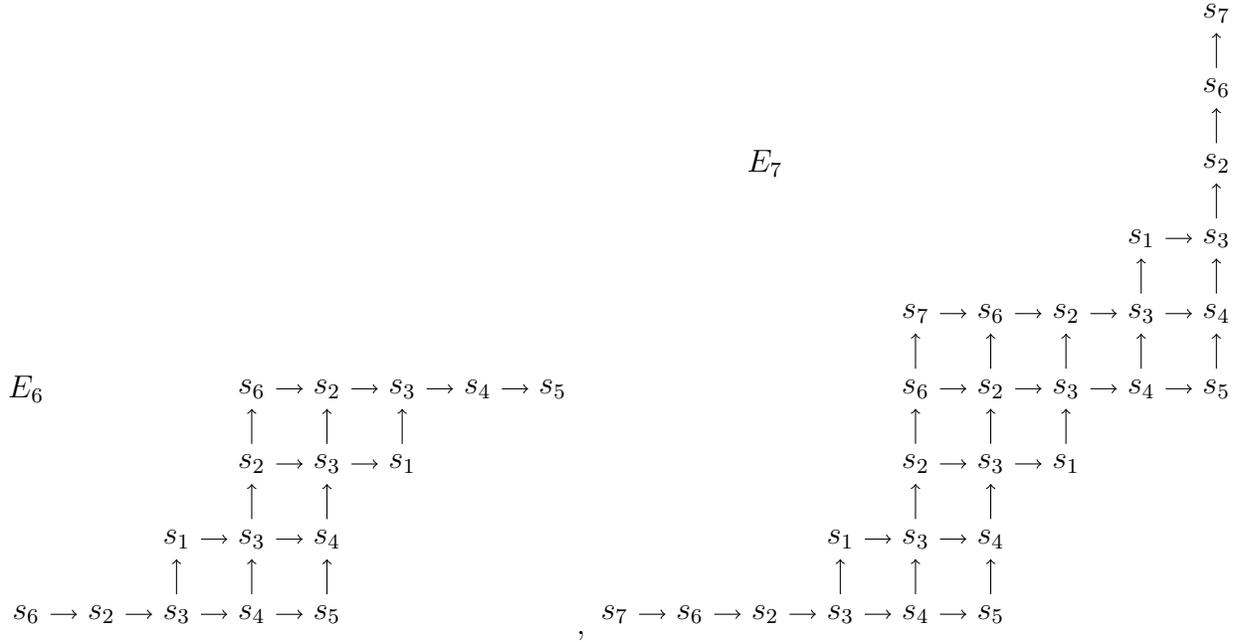
\[
        \begin{tikzpicture}
        \node (0) at (0,3) {$E_6$};
        \node (1) at (0,0) {$s_6$};
       \node (2) at (1,0) {$s_2$};
       \node (3) at (2,0) {$s_3$};
       \node (4) at (2,1) {$s_1$};
       \node (5) at (3,0) {$s_4$};
       \node (6) at (4,0) {$s_5$};
       \node (7) at (3,1) {$s_3$};
       \node (8) at (4,1) {$s_4$};
       \node (9) at (3,2) {$s_2$};
        \node (10) at (4,2) {$s_3$};
       \node (11) at (5,2) {$s_1$};
       \node (12) at (3,3) {$s_6$};
       \node (13) at (4,3) {$s_2$};
       \node (14) at (5,3) {$s_3$};
       \node (15) at (6,3) {$s_4$};
       \node (16) at (7,3) {$s_5$};

                        \draw [->] (1) to (2);
  \draw [->] (2) to (3);
  \draw [->] (3) to (4);
  \draw [->] (3) to (5);
  \draw [->] (4) to (7);
      \draw [->] (5) to (6);
  \draw [->] (5) to (7);
  \draw [->] (6) to (8);
  \draw [->] (7) to (8);
  \draw [->] (7) to (9);
      \draw [->] (8) to (10);
  \draw [->] (9) to (10);
  \draw [->] (9) to (12);
  \draw [->] (10) to (11);
  \draw [->] (10) to (13);
      \draw [->] (11) to (14);
  \draw [->] (12) to (13);
  \draw [->] (13) to (14);
  \draw [->] (14) to (15);
  \draw [->] (15) to (16);

                        \end{tikzpicture},
                                \begin{tikzpicture}
                                \node (0) at (3,7) {$E_7$};
        \node (11) at (1,1) {$s_7$};
       \node (12) at (2,1) {$s_6$};
       \node (13) at (3,1) {$s_2$};
       \node (14) at (4,1) {$s_3$};
       \node (15) at (5,1) {$s_4$};
       \node (16) at (6,1) {$s_5$};

       \node (24) at (4,2) {$s_1$};
       \node (25) at (5,2) {$s_3$};
       \node (26) at (6,2) {$s_4$};

        \node (35) at (5,3) {$s_2$};
       \node (36) at (6,3) {$s_3$};
       \node (37) at (7,3) {$s_1$};

       \node (45) at (5,4) {$s_6$};
       \node (46) at (6,4) {$s_2$};
       \node (47) at (7,4) {$s_3$};
       \node (48) at (8,4) {$s_4$};
       \node (49) at (9,4) {$s_5$};

       \node (55) at (5,5) {$s_7$};
       \node (56) at (6,5) {$s_6$};
       \node (57) at (7,5) {$s_2$};
       \node (58) at (8,5) {$s_3$};
       \node (59) at (9,5) {$s_4$};

       \node (68) at (8,6) {$s_1$};
       \node (69) at (9,6) {$s_3$};

       \node (79) at (9,7) {$s_2$};
       \node (89) at (9,8) {$s_6$};
       \node (99) at (9,9) {$s_7$};

        \draw [->] (11) to (12);
        \draw [->] (12) to (13);
        \draw [->] (13) to (14);
        \draw [->] (14) to (15);
        \draw [->] (15) to (16);

        \draw [->] (24) to (25);
        \draw [->] (25) to (26);

        \draw [->] (35) to (36);
        \draw [->] (36) to (37);

        \draw [->] (45) to (46);
        \draw [->] (46) to (47);
        \draw [->] (47) to (48);
        \draw [->] (48) to (49);

        \draw [->] (55) to (56);
        \draw [->] (56) to (57);
        \draw [->] (57) to (58);
        \draw [->] (58) to (59);

        \draw [->] (68) to (69);

        \draw [->] (14) to (24);

        \draw [->] (15) to (25);
        \draw [->] (25) to (35);
        \draw [->] (35) to (45);
        \draw [->] (45) to (55);

        \draw [->] (16) to (26);
        \draw [->] (26) to (36);
        \draw [->] (36) to (46);
        \draw [->] (46) to (56);

        \draw [->] (37) to (47);
        \draw [->] (47) to (57);

        \draw [->] (48) to (58);
        \draw [->] (58) to (68);

        \draw [->] (49) to (59);
        \draw [->] (59) to (69);
        \draw [->] (69) to (79);
        \draw [->] (79) to (89);
        \draw [->] (89) to (99);

                        \end{tikzpicture}
\]
With a small calculation, one can write down all $27$ ($56$, resp.) minimal representatives.

\subsection{Type G}

This is relatively easier. We include it here for completeness.

Any enumeration is a good enumeration for $G_2$. Let $\Delta=\{\al_1,\al_2\}$ be a set of simple roots for $G_2$, where $\al_2$ is the long root. If $\omega=\omega_1$, then  the longest element is $\tau_\ell=s_1s_2s_1s_2s_1$ and the graph $T_{\tau_\ell}$ is
\[
s_1\to s_2\to s_1\to s_2\to s_1.
\]
If $\omega=\omega_2$, then the longest element is $\tau_\ell=s_2s_1s_2s_1s_2$ and the graph $T_{\tau_{\ell}}$ is
\[
 s_2\to s_1\to s_2\to s_1\to s_2.
\]



\bibliographystyle{alpha-abbrvsort}



\end{document}